\documentclass[pdflatex,sn-mathphys-num]{sn-jnl}

\usepackage{amsmath,amssymb,amsfonts}%
\usepackage{amsthm}%
\usepackage{mathrsfs}%
\usepackage{mathtools}%
\usepackage[title]{appendix}%
\usepackage{xcolor}%
\usepackage{textcomp}%
\usepackage{manyfoot}%
\usepackage{booktabs}%
\usepackage{tikz}%
\usetikzlibrary{arrows.meta}
\usepackage{tikz-cd}

\theoremstyle{thmstyleone}%
\newtheorem{theorem}{Theorem}[section]%
\newtheorem{proposition}[theorem]{Proposition}%
\newtheorem{lemma}[theorem]{Lemma}%
\newtheorem{corollary}[theorem]{Corollary}%
\newtheorem{fact}[theorem]{Fact}%
\newtheorem*{mainthm}{Main Theorem}%

\theoremstyle{thmstylethree}%
\newtheorem{definition}[theorem]{Definition}%
\newtheorem{convention}[theorem]{Convention}%
\newtheorem{remark}[theorem]{Remark}%

\numberwithin{equation}{section}%
\newcommand{\CC}{\mathbb C}
\newcommand{\Z}{\mathbb Z}
\newcommand{\one}{\mathbf 1}
\newcommand{\id}{\mathrm{id}}
\newcommand{\Id}{\mathrm{Id}}
\newcommand{\cA}{\mathcal A}
\newcommand{\cB}{\mathcal B}
\newcommand{\cC}{\mathcal C}
\newcommand{\cD}{\mathcal D}
\newcommand{\cE}{\mathcal E}
\newcommand{\cF}{\mathcal F}

\newcommand{\cM}{\mathcal M}

\newcommand{\cV}{\mathcal V}
\newcommand{\sVec}{\mathbf{sVect}}
\newcommand{\Vect}{\mathbf{Vect}}
\DeclareMathOperator{\Hom}{Hom}
\DeclareMathOperator{\End}{End}
\DeclareMathOperator{\Aut}{Aut}
\DeclareMathOperator{\Inv}{Inv}
\DeclareMathOperator{\Irr}{Irr}
\DeclareMathOperator{\Fun}{Fun}

\DeclareMathOperator{\FPdim}{FPdim}

\newcommand{\AutLR}{\operatorname{Aut}_{\mathrm{LR}}}
\newcommand{\AutLRu}{\operatorname{Aut}^{\mathrm u}_{\mathrm{LR}}}

\newcommand{\mcl}[1]{[\![#1]\!]}
\newcommand{\Zt}{\mathcal Z_2}

\begin{document}

\title[Uniqueness of symmetry lifts]{Lifts of symmetries to minimal nondegenerate extensions are unique up to fermion attachment}

\author*[1]{\fnm{Ran} \sur{Luo}}\email{ranluo@pku.edu.cn}

\affil*[1]{\orgdiv{School of Physics}, \orgname{Peking University}, \orgaddress{\city{Beijing}, \country{China}}}

\author*[2]{\fnm{Jiahua} \sur{Tian}}\email{jtian1905@gmail.com}

\affil*[2]{\orgdiv{School of Physics}, \orgname{East China Normal University}, \orgaddress{\city{Shanghai}, \country{China}}}

\abstract{\unboldmath
Gauging fermion parity in a $(2+1)$-dimensional fermionic topological order requires lifting its symmetries to the resulting bosonic theory. We give a direct tensor-categorical proof that, for a slightly degenerate braided fusion category $\cC$ and a fixed minimal nondegenerate extension $\cE$, the restriction homomorphism $r\colon\AutLR(\cE)\to\AutLR(\cC)$ on classes of complex-linear locality-respecting (LR) braided autoequivalences is surjective with central kernel $\Z/2$, generated by fermion attachment. Thus each symmetry class has exactly two LR lift classes, related by attaching the fermion to every vortex. We apply graded Morita theory to reduce the kernel to two module types on the vortex sector, corresponding to tensoring with the vacuum or the fermion. We then prove relative rigidity to show that each type determines a unique LR class, which are distinguished by fermionic self-braiding even when their actions on simple-object labels coincide. We prove surjectivity by extending graph algebras in $\mathcal C^{\mathrm{rev}}\boxtimes\mathcal C$ to Lagrangian algebras in $\mathcal E^{\mathrm{rev}}\boxtimes\mathcal E$. Our results require neither unitarity nor a ribbon structure. We illustrate the lifting ambiguity in the Laughlin state, the $SO(3)_3$ theory and the Ising topological order.
}



\maketitle
\clearpage
\setcounter{tocdepth}{2}
\tableofcontents
\clearpage

\section{Introduction}\label{sec:intro}

A $(2+1)$-dimensional fermionic topological order is described, up to invertible phases, by a \emph{super-modular category}, which is a unitary braided fusion category $\cC$ whose M\"uger center is $\sVec=\{\one,\psi\}$, generated by the physical local fermion $\psi$. Gauging fermion parity produces a \emph{minimal modular extension} $\cC\subset\cE$, which is a modular category with $\FPdim(\cE)=2\FPdim(\cC)$ in which $\psi$ is no longer transparent. The new simple objects braiding nontrivially with $\psi$ are the \emph{parity vortices}. More generally, a braided fusion category $\cC$ is \emph{slightly degenerate} if its M\"uger center is braided equivalent to $\sVec$~\cite{Davydov:2011dqx}, and a \emph{minimal nondegenerate extension} of $\cC$ is a nondegenerate braided fusion category $\cE$ containing $\cC$ as a braided fusion subcategory, with $\FPdim(\cE)=2\FPdim(\cC)$. Johnson-Freyd and Reutter proved that every slightly degenerate braided fusion category has a minimal nondegenerate extension \cite[Main Theorem]{Johnson-Freyd:2021chu}. For super-modular categories, and more generally for pseudo-unitary slightly degenerate ones, this extension can be taken pseudo-unitary, hence modular with its positive ribbon structure \cite[\S1.3]{Johnson-Freyd:2021chu}.

A global symmetry of the fermionic phase acts on anyons by a braided autoequivalence of $\cC$~\cite[\S III.B, Eq.~(80)]{Barkeshli:2014cna},~\cite[Eq.~(24) and \S IV.A]{Bulmash:2021hmb}, which is constrained by locality. Namely, the symmetry must act trivially on the fusion space $\Hom(\psi\otimes\psi,\one)$ of the local fermion~\cite[\S IV.A, Eq.~(59)]{Bulmash:2021hmb}, and two actions are identified only through natural isomorphisms acting trivially on $\psi$~\cite[\S IV.B]{Bulmash:2021hmb}. These two conditions define the group $\AutLR(\cC)$ of \emph{locality-respecting} (LR) autoequivalences \cite[\S IV.C, Eqs.~(65) and~(77)]{Bulmash:2021ryq}.

Gauging fermion parity $\Z_2^f\subset G_f$ while preserving the remaining symmetry $G_b=G_f/\Z_2^f$~\cite[\S IV.A]{Bulmash:2021hmb} produces the bosonic theory $\cE$, on which $G_b$ must still act. Each symmetry of $\cC$ must therefore be lifted to a symmetry of $\cE$ that restricts to it, and such lifts also enter the computation of anomalies~\cite{Bulmash:2021ryq}. The lifts of a class are governed by the kernel of the restriction homomorphism
\begin{equation}\label{eq:intro-r}
	r\colon\AutLR(\cE)\longrightarrow\AutLR(\cC).
\end{equation} 
The kernel always contains a subgroup $\Z/2$ generated by the class of \emph{fermion attachment} $\alpha$, which is the identity on $\cC$ and sends a parity vortex $X$ to $X\otimes\psi$. It is called the vortex permuting symmetry in~\cite[\S III.A, Eqs.~(111)--(115)]{Aasen:2021vva}; see also~\cite[\S VI.B.2, Eqs.~(104)--(108)]{Bulmash:2021ryq}. It is further conjectured that nothing else is in the kernel~\cite[Conj.~VI.1]{Bulmash:2021ryq}, namely:
\begin{equation}\label{eq:BBconj}
\ker r\cong\Z/2\qquad\text{for all super-modular $\cC$ and all minimal modular extensions $\cE$.}
\end{equation}

Aasen, Bonderson, and Knapp proved~\eqref{eq:BBconj} in the skeletal formalism when $\cE$ contains a \emph{superAbelian} vortex, that is, an Abelian vortex or a vortex $X$ of dimension $\sqrt2$ with $X\otimes\psi\cong X$~\cite[App.~F.1--F.2]{Aasen:2021vva}, and~\eqref{eq:BBconj} also enters later work on anomalies of fermionic topological phases~\cite[App.~C, Conj.~C.3]{Debray:2023iwf},~\cite[App.~D]{Zhou:2024fkv}. So far, \eqref{eq:BBconj} has been proved directly only in the superAbelian case~\cite[App.~F.3]{Aasen:2021vva},~\cite[\S VI.B.5]{Bulmash:2021ryq}, although $|\ker r|=2$ can also be deduced from the higher-categorical results of~\cite[\S3 and Rem.~3.1]{Johnson-Freyd:2026galois},~\cite{Johnson-Freyd:2021chu}.

Our contribution is a direct tensor-categorical proof of~\eqref{eq:BBconj}, in the generality of slightly degenerate braided fusion categories and their minimal nondegenerate extensions, together with surjectivity of~$r$, so that $\AutLR(\cE)$ is a central extension of $\AutLR(\cC)$ by $\Z/2$. In the proof, we construct explicit LR representatives and control the full tensor structure. We use two ingredients, the classification of the module type (Proposition~\ref{prop:module-type}) and relative rigidity (Lemma~\ref{lem:rigidity}), which we outline in Section~\ref{subsec:intro-idea}. With them, we treat extensions with and without superAbelian vortices uniformly. Since relative rigidity holds for arbitrary faithfully $\Z/2$-graded braided fusion categories, we can use it to exclude hidden tensor structures even when some fusion multiplicities exceed one. We use graph algebras to prove surjectivity by enlarging the graph algebra of a symmetry of $\cC$ to a Lagrangian algebra in $\cE^{\mathrm{rev}}\boxtimes\cE$, which is the graph of a lift to $\cE$ (Proposition~\ref{prop:surjectivity}). 

Physically, every linear LR symmetry class has exactly two lifts to the parity vortices. Given one lift, the other attaches $\psi$ throughout the vortex sector. The fermionic sign $c_{\psi,\psi}=-\id$ distinguishes the two LR classes even when their permutations of simple objects coincide. We show that this sign appears in the tensor structure $J^\alpha_{X,\psi}\colon\alpha(X)\otimes\alpha(\psi)\to\alpha(X\otimes\psi)$ of fermion attachment, which equals $-\id$ for every vortex $X$ because the fermion attached to $X$ is exchanged with $\psi$. Consequently every monoidal natural isomorphism $\eta\colon\alpha\Rightarrow\Id_\cE$ satisfies $\eta_\psi=-\id_\psi$, so $\alpha$ is not identified with $\Id_\cE$ by a natural isomorphism acting trivially on $\psi$ (Proposition~\ref{prop:alpha-nontrivial}).

On the mathematical side, in Proposition~\ref{prop:G1} we identify the LR kernel with the automorphism group of $\cE$ in the groupoid of marked minimal nondegenerate extensions of $\cC$, and we use the same argument to compute its group of 2-automorphisms. These explicit calculations realize the homotopy-group description in \cite[\S3]{Johnson-Freyd:2026galois}. They fit the theory of braided graded extensions and braided Picard 2-groups in \cite[\S8]{Davydov-Nikshych:2021}. However, in the kernel proof we use only graded Morita theory and ordinary natural transformations.

\subsection{Main results}\label{subsec:intro-results}

We work with slightly degenerate braided fusion categories and their minimal nondegenerate extensions, examples of which are super-modular categories and their minimal modular extensions. We prove the following statement.

\begin{mainthm}
Let $\cC$ be a slightly degenerate braided fusion category with fermion $\psi$, and fix a minimal nondegenerate extension $\cC\subset\cE$. The restriction homomorphism $r\colon\AutLR(\cE)\to\AutLR(\cC)$ of $\CC$-linear locality-respecting braided autoequivalences fits into the central short exact sequence
\begin{equation}\label{eq:central-exact}
	1\longrightarrow\Z/2\xrightarrow{\;1\mapsto[\alpha_\psi]\;}\AutLR(\cE)\xrightarrow{\;r\;}\AutLR(\cC)\longrightarrow1,
\end{equation}
where $[\alpha_\psi]$ is the fermion-attachment class. In particular, $r$ is surjective, $\ker r=\{[\Id_\cE,\id_\psi],[\alpha_\psi]\}$, and every class in $\AutLR(\cC)$ has exactly two lifts, which differ by $[\alpha_\psi]$. If $\cE$ is unitary, the same central short exact sequence exists with $\AutLR$ replaced by $\AutLRu$, the groups of unitary locality-respecting autoequivalences modulo unitary equivalence.
\end{mainthm}

Here $\AutLR(\cE)$ is the group of classes $[F,\varphi]$ of pairs consisting of a braided autoequivalence $F$ of $\cE$ and an isomorphism $\varphi\colon F(\psi)\to\psi$ satisfying the local-vertex condition, up to monoidal natural isomorphisms compatible with $\varphi$ (Definition~\ref{def:LR}). The class $[\alpha_\psi]$ is that of $(\alpha,\id_\psi)$, where $\alpha$ is fermion attachment (Definition~\ref{def:alpha}). It is the identity tensor functor on $\cC$ and sends a vortex $X$ to $X\otimes\psi$. For vortices $X,Y$, the tensor structure
\begin{equation*}
	J^\alpha_{X,Y}\colon\alpha(X)\otimes\alpha(Y)=(X\otimes\psi)\otimes(Y\otimes\psi)\longrightarrow X\otimes Y=\alpha(X\otimes Y)
\end{equation*}
of $\alpha$ braids the fermion attached to $X$ past $Y$ by $c_{\psi,Y}$ and then applies a fixed isomorphism $e\colon\psi\otimes\psi\to\one$ to the two attached fermions.

\subsection{Idea of the proof and organization of the paper}\label{subsec:intro-idea}
Let $[F,\varphi]\in\ker r$, and let $\cE=\cC\oplus\cV$ be the grading by monodromy with $\psi$, whose odd part $\cV$ is the vortex sector (Lemma~\ref{lem:grading}). By the definition of $r$, there is a monoidal natural isomorphism $\eta\colon F|_\cC\Rightarrow\Id_\cC$ with $\eta_\psi=\varphi$. We extend $\eta$ by the identity on $F|_\cV$ to a natural isomorphism $\Theta\colon F\Rightarrow\Phi$ and transport the tensor structure along $\Theta$ (Lemma~\ref{lem:normalization}). This is a change of representative within the same LR class and can be understood physically as making a ``gauge choice'' as follows
\begin{equation*}
\begin{tikzpicture}[baseline=(current bounding box.center)]
\node (source-left) at (0,0) {$\cE$};
\node (target-left) at (0,-1.8) {$\cE$};
\node (source-right) at (5.2,0) {$\cE$};
\node (target-right) at (5.2,-1.8) {$\cE$};
\draw[->] (source-left) -- node[left=5pt] {$(F,\varphi)$} (target-left);
\draw[->] (source-right) -- node[right=5pt] {$(\Phi,\id_\psi)$} (target-right);
\draw[double,double equal sign distance,-{Implies[length=6pt,width=6pt]}]
  (0.35,-0.9) -- node[above=5pt] {gauge choice}
  node[below=5pt] {$\Theta=\eta\oplus\id_{F|_\cV}$} (4.85,-0.9);
\end{tikzpicture}
\end{equation*}
Here $\Phi=\Id_\cC\oplus F|_\cV$ as a linear functor, equipped with the transported tensor structure $J^\Phi$, and $\Phi|_\cC=\Id_\cC$ as a tensor functor. Replacing $(F,\varphi)$ by this representative, we may therefore assume that $F$ is literally the identity tensor functor on $\cC$ and that $\varphi=\id_\psi$. The question is then how much freedom $F$ retains on $\cV$. The answer is a single binary choice, and we prove this using two facts.

The first fact is that \emph{$F$ acts on the vortex sector as $-\otimes\one$ or as $-\otimes\psi$}. Since $F$ is the identity on $\cC$, its restriction to $\cV$, together with the mixed tensor structure $J^F_{b,X}$ for $b\in\cC$ and $X\in\cV$, is a left $\cC$-module autoequivalence of $\cV$. By graded Morita theory~\cite{Etingof:2009yvg}, $\cV$ is an invertible $\cC$-bimodule category, so every such functor is isomorphic to a right multiplication $R_a=-\otimes a$ by a single object $a\in\cC$ (Corollary~\ref{cor:morita}). This is where the modular data are too coarse, namely they may split the vortices into blocks that seem to admit independent choices, but the module structure couples all vortices through their fusion with $\cC$, and there is only one global choice $a$. Since $F$ is braided, $a$ must centralize $\cC$, so $a$ lies in the M\"uger center $\Zt(\cC)=\langle\psi\rangle$, and since $F$ is fully faithful, $a\cong\one$ or $a\cong\psi$ (Proposition~\ref{prop:module-type}). The value $a=\psi$ is realized by fermion attachment $\alpha$ (Proposition~\ref{prop:alpha}).

The second fact is \emph{relative rigidity}, namely the kernel element is determined by the value of $a$. This requires control of the tensor structure and of the allowed natural isomorphisms: fixing every simple object up to isomorphism does not imply that an LR class is trivial, as fermion attachment in the Ising theory illustrates (Proposition~\ref{prop:alpha-nontrivial}). For a kernel element with $a=\one$, after a further transport to a tensor structure $J'$, the only data not controlled by the module functor are the vortex--vortex components $J'_{X,Y}$, which act on the fusion spaces $\Hom(X\otimes Y,b)$ with $b\in\cC$. These components satisfy
\begin{equation*}
	\id_{X\otimes Y}\otimes J'_{Z,W}=J'_{X,Y}\otimes\id_{Z\otimes W}\qquad\text{for all vortices }X,Y,Z,W,
\end{equation*}
which follows from three instances of the monoidal axiom~\eqref{eq:monoidal-axiom} for $J'$. So $J'_{X,Y}=\lambda\,\id_{X\otimes Y}$ for a single scalar $\lambda$, and rescaling all vortices by $\sqrt\lambda$ removes it (Lemma~\ref{lem:rigidity}). Hence a kernel element with $a=\one$ is trivial, by an isomorphism that is the identity on $\cC$ and in particular at $\psi$. By composing with $\alpha$, we reduce the case $a=\psi$ to the case $a=\one$, so $\ker r\subseteq\{1,[\alpha_\psi]\}$.

Conversely, it is known for super-modular categories in the skeletal formulation that fermion attachment generates a subgroup $\{1,[\alpha_\psi]\}$ of order two of $\ker r$, namely $\{1,[\alpha_\psi]\}\subseteq\ker r$~\cite{Aasen:2021vva,Bulmash:2021ryq}. We include a short tensor-categorical proof for two reasons. First, it applies to every slightly degenerate $\cC$, unitary or not, and second, in the proof of $\ker r\subseteq\{1,[\alpha_\psi]\}$ we use the specific representative $(\alpha,\id_\psi)$, with $\alpha|_\cC=\Id_\cC$ and $\alpha|_\cV=R_\psi$ as module functors. The class $[\alpha_\psi]$ is nontrivial. Indeed, since the fermionic self-braiding is $c_{\psi,\psi}=-\id$, every monoidal natural isomorphism $\eta\colon\alpha\Rightarrow\Id_\cE$ has $\eta_\psi=-\id_\psi$, whereas a morphism of pairs $(\alpha,\id_\psi)\to(\Id_\cE,\id_\psi)$ requires $\eta_\psi=\id_\psi$ (Proposition~\ref{prop:alpha-nontrivial}). This sign is essential. In the Ising extension of $\sVec$, $\alpha$ fixes every simple object up to isomorphism, but $(\alpha,\id_\psi)\not\cong(\Id_\cE,\id_\psi)$ (Proposition~\ref{prop:alpha-nontrivial}). Applying relative rigidity to $\alpha^2$, we get $[\alpha_\psi]^2=1$ (Corollary~\ref{cor:alpha-square}). Combining this with $\ker r\subseteq\{1,[\alpha_\psi]\}$, we obtain $\ker r=\{1,[\alpha_\psi]\}\cong\Z/2$.

We prove surjectivity separately in Proposition~\ref{prop:surjectivity}. For a braided autoequivalence $F$ of $\cC$, we condense its graph algebra in $\cE^{\mathrm{rev}}\boxtimes\cE$. The resulting local-module category has dimension four, contains a fermion, and is Witt-trivial, so it is the toric-code category. Condensing one of its bosons enlarges the graph algebra to a Lagrangian algebra. Using the parity grading and the classification of connected \'etale algebras in a product, we then identify this algebra with the graph of a braided autoequivalence of $\cE$ restricting to $F$. We use neither unitarity nor a ribbon structure in this argument.

The work is organized as follows. In Section~\ref{sec:prelim}, we fix the monoidal conventions and recall the parity grading and graded Morita theory. In Section~\ref{sec:LR}, we define LR autoequivalences and their restriction map, and we compare them with the skeletal formulation of the physics literature. In Sections~\ref{sec:vortex} and~\ref{sec:attachment}, we establish the module-type classification, relative rigidity, and the properties of fermion attachment. In Section~\ref{sec:proof}, we prove surjectivity and complete the proof of the Main Theorem. In Section~\ref{sec:extension}, we identify the kernel with the automorphism group of a marked minimal extension. In Section~\ref{sec:applications}, we interpret the kernel theorem physically in the Laughlin state, the $SO(3)_3$ theory and the Ising theory. In Appendix~\ref{app:square}, we give an explicit trivialization of $\alpha^2$.

\section{Conventions and preliminaries}\label{sec:prelim}

We work over $\CC$ and use the terminology of \cite{EGNO:2015} for fusion categories, tensor functors, module categories and braided categories. We write $\Irr(\cA)$ and $\Inv(\cA)\subset\Irr(\cA)$ for the sets of isomorphism classes of simple and of invertible objects. We adopt the following conventions.

\begin{convention}\label{conv:coherence}
All monoidal categories are understood in the usual weak, not necessarily strict, sense. We write
\begin{equation*}
	a_{X,Y,Z}\colon (X\otimes Y)\otimes Z\longrightarrow X\otimes(Y\otimes Z), \quad l_X\colon\one\otimes X\longrightarrow X, \quad r_X\colon X\otimes\one\longrightarrow X
\end{equation*}
for their associator and left and right unit constraints (unitors). These are part of the data and satisfy the pentagon and triangle axioms. A full monoidal subcategory carries the inherited constraints. We use superscripts to distinguish categories when necessary.

By Mac Lane's coherence theorem~\cite[Thm.~2.9.2 and Rem.~2.9.3]{EGNO:2015}, between two bracketings of a word in objects, possibly with copies of $\one$ inserted, there is a unique isomorphism composed of tensor products of $a$, $l$, $r$, their inverses and identities. We call it the canonical isomorphism. When we omit $\otimes$ and canonical isomorphisms in a diagram, an object stands for a word, in any bracketing. An arrow $usv\to utv$ labelled $f\colon s\to t$, where $u$ and $v$ are possibly empty words, stands for the canonical isomorphism from the given bracketing of $usv$ to one in which $s$ is a single factor, followed by $\id_u\otimes f\otimes\id_v$, followed by the canonical isomorphism to the given bracketing of $utv$. By coherence and naturality of $a$, $l$ and $r$, this morphism does not depend on the bracketing in which $s$ is a single factor. Hence whether such a diagram commutes does not depend on the bracketings chosen for its objects.
\end{convention}

\begin{convention}\label{conv:graded}
If $\cE=\cE_0\oplus\cE_1$ as abelian categories~\cite[\S1.3]{EGNO:2015}, we call the objects of $\cE_0\cup\cE_1$ homogeneous. Linear functors, natural transformations and tensor structures are specified on homogeneous objects and morphisms and extended by additivity. Since every object is a direct sum of homogeneous ones, morphisms between $\cE_0$ and $\cE_1$ vanish, and all functors involved are additive, with $\otimes$ additive in each variable. Naturality, the monoidal axiom and the braiding compatibility need only be checked on homogeneous objects.
\end{convention}

\begin{convention}\label{conv:unitary}
Unitary fusion categories carry a positive $\dagger$-structure, unitary braided fusion categories also have unitary braiding, and unitary tensor functors are $\dagger$-functors with unitary tensor structure. However, unitarity plays no role in the proofs, and it enters only through the statement of the Main Theorem for the groups $\AutLRu$ of unitary classes, which is needed for the comparison with~\cite{Bulmash:2021ryq} and which we deduce from the general case.
\end{convention}

\subsection{Tensor functors and their transport}
\label{subsec:tensor-functors}

We use strong monoidal functors following the conventions of~\cite[\S2.4]{EGNO:2015}. A tensor functor $F\colon\cA\to\cB$ has structure isomorphisms
\begin{equation*}
	J^F_{X,Y}\colon F(X)\otimes F(Y)\longrightarrow F(X\otimes Y), \qquad \iota^F\colon\one_\cB\longrightarrow F(\one_\cA).
\end{equation*}
Writing composition from right to left, $J^F$ and $\iota^F$ satisfy the \emph{monoidal axiom}
\begin{equation}\label{eq:monoidal-axiom}
	F(a^\cA_{X,Y,Z})J^F_{X\otimes Y,Z}(J^F_{X,Y}\otimes\id_{F(Z)}) =J^F_{X,Y\otimes Z}(\id_{F(X)}\otimes J^F_{Y,Z})a^\cB_{F(X),F(Y),F(Z)}
\end{equation}
and the \emph{unit conditions}
\begin{equation}\label{eq:functor-units}
\begin{aligned}
	F(l^\cA_X)J^F_{\one,X}(\iota^F\otimes\id_{F(X)})&=l^\cB_{F(X)},\\
	F(r^\cA_X)J^F_{X,\one}(\id_{F(X)}\otimes\iota^F)&=r^\cB_{F(X)}.
\end{aligned}
\end{equation}
Here $a$, $l$ and $r$ are the associator and unit constraints of Convention~\ref{conv:coherence}, and we omit the superscripts $\cA$ and $\cB$ when they are clear.

Up to monoidal isomorphism, we may choose $F$ to be \emph{unit-normalized}, meaning $F(\one_\cA)=\one_\cB$ and $\iota^F=\id_\one$~\cite[Rem.~2.4.6]{EGNO:2015}, in which case we have
\begin{equation}\label{eq:normalized-unit-tensorators}
	J^F_{\one,X}=F(l_X)^{-1}l_{F(X)}, \qquad J^F_{X,\one}=F(r_X)^{-1}r_{F(X)},
\end{equation}
which are maps $\one\otimes F(X)\to F(\one\otimes X)$ and $F(X)\otimes\one\to F(X\otimes\one)$, and so are identities only when the unit constraints are strict. We use unit-normalized representatives for LR autoequivalences.

For composable tensor functors $G\colon\cA\to\cB$ and $F\colon\cB\to\cD$, the composite $FG=F\circ G$ has structure
\begin{equation}\label{eq:composite}
	J^{FG}_{X,Y}=F(J^G_{X,Y})J^F_{G(X),G(Y)},\qquad \iota^{FG}=F(\iota^G)\iota^F.
\end{equation}
For parallel tensor functors $F,G\colon\cA\to\cB$, a natural transformation $\eta\colon F\Rightarrow G$ is \emph{monoidal} if
\begin{equation}\label{eq:monoidal-nat}
	\eta_{X\otimes Y}J^F_{X,Y}=J^G_{X,Y}(\eta_X\otimes\eta_Y),\qquad \eta_\one\iota^F=\iota^G
\end{equation}
which are the tensor and unit conditions, respectively~\cite[Def.~2.4.8 and Rem.~2.4.9]{EGNO:2015}. For unit-normalized functors the unit condition becomes $\eta_\one=\id_\one$. A braiding $c$ satisfies the hexagon axioms~\cite[\S8.1]{EGNO:2015}, which in the notation of Convention~\ref{conv:coherence} read $c_{X,Y\otimes Z}=(\id_Y\otimes c_{X,Z})(c_{X,Y}\otimes\id_Z)$ and $c_{X\otimes Y,Z}=(c_{X,Z}\otimes\id_Y)(\id_X\otimes c_{Y,Z})$. A tensor functor between braided categories is \emph{braided} if
\begin{equation}\label{eq:braided-functor}
	F(c^\cA_{X,Y})J^F_{X,Y}=J^F_{Y,X}c^\cB_{F(X),F(Y)}
\end{equation}
for all $X,Y$ \cite[Def.~8.1.7]{EGNO:2015}.

\begin{lemma}\label{lem:transport}
Let $(F,J^F,\iota^F)\colon\cA\to\cB$ be a tensor functor, let $G\colon\cA\to\cB$ be a functor, and let $\theta\colon F\Rightarrow G$ be a natural isomorphism. There is a unique tensor structure on $G$ making $\theta$ monoidal, namely
\begin{equation}\label{eq:transport}
\begin{aligned}
	J^G_{X,Y} &=\theta_{X\otimes Y}J^F_{X,Y}
       (\theta_X^{-1}\otimes\theta_Y^{-1}),\\
       \iota^G&=\theta_\one\iota^F.
\end{aligned}
\end{equation}
If the categories are braided and $F$ is braided, then $G$ is braided as well.
\end{lemma}

\begin{proof}
The formulas are forced by the defining conditions~\eqref{eq:monoidal-nat}, cf.~\cite[Def.~2.4.8 and Rem.~2.4.9]{EGNO:2015}. We substitute~\eqref{eq:transport} into the monoidal axiom~\eqref{eq:monoidal-axiom} and the unit conditions~\eqref{eq:functor-units} for $(G,J^G,\iota^G)$. By naturality of $\theta$ and of the constraints of $\cB$, the components of $\theta$ cancel, and both conditions reduce to those for $F$. The braidedness condition follows similarly from~\eqref{eq:braided-functor} and naturality of the braiding.
\end{proof}

We will also use the following elementary cancellation property.

\begin{lemma}\label{lem:faithful}
Let $A,B$ be nonzero objects of a fusion category.
\begin{enumerate}
\item If $g\in\End(B)$ satisfies $\id_A\otimes g=0$ or $g\otimes\id_A=0$, then $g=0$.
\item If $f\in\End(A)$ and $g\in\End(B)$ satisfy $f\otimes\id_B=\id_A\otimes g$, then there is $\lambda\in\CC$ such that $f=\lambda\id_A$ and $g=\lambda\id_B$.
\end{enumerate}
\end{lemma}

\begin{proof}
We choose a left dual of $A$, with evaluation $\mathrm{ev}_A\colon A^*\otimes A\to\one$. By the rigidity axioms, $\mathrm{ev}_A\neq0$. By semisimplicity and simplicity of $\one$, there is therefore a section $s\colon\one\to A^*\otimes A$ with $\mathrm{ev}_A s=\id_\one$. For $h\in\End(A\otimes B)$, set
\begin{equation*}
	\pi(h):=l_B(\mathrm{ev}_A\otimes\id_B)a^{-1}_{A^*,A,B}(\id_{A^*}\otimes h)a_{A^*,A,B}(s\otimes\id_B)l_B^{-1}.
\end{equation*}
By naturality of the constraints, for $f\in\End(A)$ and $g\in\End(B)$,
\begin{equation}\label{eq:pi-cancel}
	\pi(\id_A\otimes g)=g,\qquad\pi(f\otimes\id_B)=\lambda_f\id_B,\qquad\lambda_f:=\mathrm{ev}_A(\id_{A^*}\otimes f)s\in\End(\one)=\CC.
\end{equation}

(1) If $\id_A\otimes g=0$, then $g=\pi(\id_A\otimes g)=0$ by~\eqref{eq:pi-cancel}. The case $g\otimes\id_A=0$ follows in the same way from a right dual of $A$, applied to $\End(B\otimes A)$.

(2) Applying $\pi$ to $f\otimes\id_B=\id_A\otimes g$, we get $\lambda_f\id_B=g$ by~\eqref{eq:pi-cancel}. So $(f-\lambda_f\id_A)\otimes\id_B=0$, and $f=\lambda_f\id_A$ by (1), applied with the roles of $A$ and $B$ exchanged. Hence (2) holds with $\lambda=\lambda_f$.
\end{proof}

\subsection{Fermions and slightly degenerate categories}\label{subsec:fermions}
A \emph{fermion} in a braided fusion category $\cB$ is an invertible object $\psi$ with $\psi\otimes\psi\cong\one$ and $c_{\psi,\psi}=-\id$ \cite[Def.~3.1(b)]{Galindo-Venegas-Ramirez:2025}. The fusion subcategory $\langle\psi\rangle$ is then braided equivalent to $\sVec$ \cite[Thm.~8.4.12]{EGNO:2015}. We write $\cD'$ for the centralizer of a fusion subcategory $\cD\subseteq\cB$ and $\Zt(\cB):=\cB'$ for the M\"uger center \cite[Def.~8.20.1]{EGNO:2015}. For objects $X,Y$ of $\cB$ we call the double braiding $c_{Y,X}\circ c_{X,Y}\in\Aut(X\otimes Y)$ the \emph{monodromy} of $X$ with $Y$. If $Y$ is invertible, that is, $\mathrm{ev}_Y$ and $\mathrm{coev}_Y$ are isomorphisms~\cite[Def.~2.11.1]{EGNO:2015}, then $-\otimes Y$ is an autoequivalence with quasi-inverse $-\otimes Y^*$. So if moreover $X$ is simple, then $X\otimes Y\cong Y\otimes X$ is simple, and the monodromy is a scalar, which we call the \emph{monodromy scalar}. The monodromy of $X$ with $Y$ is conjugate by $c_{X,Y}$ to the monodromy of $Y$ with $X$, and if one of them is a scalar, then so is the other with the same value. The category $\cB$ is \emph{slightly degenerate} if $\Zt(\cB)=\langle\psi\rangle$ for a fermion $\psi$~\cite{Davydov:2011dqx} which is unique up to isomorphism as the only nontrivial simple object of $\Zt(\cB)$. A super-modular category is a unitary slightly degenerate braided fusion category with its canonical ribbon structure \cite{Bruillard:2016yio}. The category $\cB$ is called almost non-degenerate if $\Zt(\cB)\in\{\Vect,\sVec\}$~\cite[Def.~3.25]{Davydov:2010kfz}.

\begin{lemma}\label{lem:mu}
Let $\psi$ be a fermion in a braided fusion category and let
$e\colon\psi\otimes\psi\to\one$ be an isomorphism. Then
\begin{equation}\label{eq:e-assoc}
l_\psi(e\otimes\id_\psi)
=r_\psi(\id_\psi\otimes e)a_{\psi,\psi,\psi},
\qquad e\circ c_{\psi,\psi}=-e.
\end{equation}
Set $\psi^0:=\one$, $\psi^1:=\psi$ and define
$\mu_{p,q}\colon\psi^p\otimes\psi^q\to\psi^{p+q}$, with exponents modulo two, by
\[
\mu_{0,q}=l_{\psi^q},\qquad \mu_{p,0}=r_{\psi^p},\qquad \mu_{1,1}=e.
\]
The two definitions of $\mu_{0,0}$ agree since $l_\one=r_\one$. These maps satisfy
\begin{subequations}\label{eq:mu-props}
\begin{align}
\mu_{p+q,r}(\mu_{p,q}\otimes\id)
&=\mu_{p,q+r}(\id\otimes\mu_{q,r})a_{\psi^p,\psi^q,\psi^r},\label{eq:mu-assoc}\\
\mu_{q,p}c_{\psi^p,\psi^q}&=(-1)^{pq}\mu_{p,q}.\label{eq:mu-comm}
\end{align}
\end{subequations}
\end{lemma}

\begin{proof}
The braided equivalence $\langle\psi\rangle\simeq\sVec$ transports the usual
associative multiplication on the two fermion powers to maps of this form,
with some nonzero $e_0\colon\psi\otimes\psi\to\one$.
Every other $e$ is a scalar multiple of $e_0$. In the only associativity
identity with three nonzero exponents, each side contains $e$ exactly once,
so it remains valid after this rescaling. Cases involving a zero exponent
are the unit axioms and naturality of $l,r$. The braiding identity~\eqref{eq:mu-comm} follows from $c_{\psi,\psi}=-\id$ and the identities $c_{\one,X}=r_X^{-1}l_X$ and $c_{X,\one}=l_X^{-1}r_X$.
\end{proof}

\subsection{Minimal nondegenerate extensions and the parity grading}\label{subsec:grading}

Let $\cC$ be slightly degenerate with fermion $\psi$. A \emph{minimal nondegenerate extension} of $\cC$ is a nondegenerate braided fusion category $\cE$ containing $\cC$ as a full braided fusion subcategory with $\FPdim(\cE)=2\FPdim(\cC)$. Its existence is established in~\cite[Main Theorem]{Johnson-Freyd:2021chu}. In the unitary setting, with its canonical ribbon structure, it is a \emph{minimal modular extension}~\cite{Lan:2016rcq,Bruillard:2016yio}. We fix $\cC\subset\cE$ throughout.

The double-braiding signs and the resulting $\Z/2$-grading are given in~\cite[Props.~2.3--2.4]{Bruillard:2016yio} in the unitary ribbon setting. We include a proof in the general braided fusion setting used here.

\begin{lemma}\label{lem:grading}
For every simple $X\in\cE$ there is $p(X)\in\Z/2$ with
\begin{equation}\label{eq:parity}
\begin{aligned}
	M_{X,\psi}&:=c_{\psi,X}\circ c_{X,\psi}=(-1)^{p(X)}\,\id_{X\otimes\psi},\\
	\text{equivalently}\quad M_{\psi,X}&:=c_{X,\psi}\circ c_{\psi,X}=(-1)^{p(X)}\,\id_{\psi\otimes X}.
\end{aligned}
\end{equation}
If $Z$ is a simple summand of $X\otimes Y$, then $p(Z)=p(X)+p(Y)$. Let $\cE_i$ be the full subcategory of direct sums of simple objects of parity $i$. Then $\cE=\cE_0\oplus\cE_1$ is a faithful $\Z/2$-grading with $\cE_0=\cC=\langle\psi\rangle'$, and $\cC'=\langle\psi\rangle$. Every braided autoequivalence $F$ of $\cE$ with $F(\psi)\cong\psi$ satisfies $F(\cE_i)\subseteq\cE_i$.
\end{lemma}

\begin{proof}
Let $X$ be simple. Since $\psi$ is invertible, $X\otimes\psi$ and $\psi\otimes X$ are simple, so both double braidings in~\eqref{eq:parity} are scalars. Since $M_{\psi,X}=c_{X,\psi}\,M_{X,\psi}\,c_{X,\psi}^{-1}$, the two scalars agree, and we write
\begin{equation*}
	M_{X,\psi}=\lambda_X\,\id_{X\otimes\psi},\qquad M_{\psi,X}=\lambda_X\,\id_{\psi\otimes X}.
\end{equation*}
By~\cite[Lemma~8.22.1(i)]{EGNO:2015}, if $M_{U,W}=\lambda\,\id_{U\otimes W}$ and $M_{V,W}=\mu\,\id_{V\otimes W}$, then $M_{U\otimes V,W}=\lambda\mu\,\id_{(U\otimes V)\otimes W}$. For $U=V=\psi$ and $W=X$, we get $M_{\psi\otimes\psi,X}=\lambda_X^2\,\id$. Since $\psi\otimes\psi\cong\one$ and $M_{\one,X}=\id$, we have $\lambda_X^2=1$ by naturality of the braiding. Hence $\lambda_X=(-1)^{p(X)}$ for a unique $p(X)\in\Z/2$. For $W=\psi$, $U=X$ and $V=Y$, we get $\lambda_Z=\lambda_X\lambda_Y$ for every simple summand $Z$ of $X\otimes Y$ by naturality of the braiding, that is, $p(Z)=p(X)+p(Y)$.

Since $\psi\in\Zt(\cC)$, we have $\langle\psi\rangle\subseteq\cC'$. As $\cE$ is nondegenerate, $\FPdim(\cC)\FPdim(\cC')=\FPdim(\cE)=2\FPdim(\cC)$~\cite[Cor.~8.21.7]{EGNO:2015}, so $\FPdim(\cC')=2=\FPdim(\langle\psi\rangle)$ and $\cC'=\langle\psi\rangle$ by~\cite[Prop.~6.3.3]{EGNO:2015}. By the double centralizer theorem~\cite[Cor.~8.21.2]{EGNO:2015}, $\cC=\cC''=\langle\psi\rangle'$. A simple object has parity $0$ if and only if it centralizes $\psi$, so $\cE_0=\langle\psi\rangle'=\cC$. If $\cE_1=0$, then $\psi$ centralizes $\cE$, so $\psi\in\Zt(\cE)=\Vect$, which contradicts $c_{\psi,\psi}=-\id$. Hence $\cE_1\neq0$, and since $\one\in\cE_0$, the grading is faithful.

Let $F$ be a braided autoequivalence of $\cE$ with $F(\psi)\cong\psi$, and let $X$ be simple. Applying~\eqref{eq:braided-functor} for $F$ twice, we get $F(c_{\psi,X}c_{X,\psi})\,J^F_{X,\psi}=J^F_{X,\psi}\,c_{F(\psi),F(X)}c_{F(X),F(\psi)}$. Hence the monodromy of $F(X)$ with $F(\psi)\cong\psi$ is $\lambda_X$, that is, $p(F(X))=p(X)$.
\end{proof}

In the physics literature, the fact that every braided autoequivalence $F$ of $\cE$ with $F(\psi)\cong\psi$ satisfies $F(\cE_i)\subseteq\cE_i$ appears in~\cite[\S III.A, Eqs.~(101)--(103)]{Aasen:2021vva}, where fermionic topological symmetries are shown not to change the vorticity of anyons.

We write $\cV:=\cE_1$ and call its simple objects \emph{vortices}. Thus $\cC\otimes\cV\subseteq\cV$, $\cV\otimes\cC\subseteq\cV$, $\cV\otimes\cV\subseteq\cC$, and $X\otimes\psi\in\cV$ for $X\in\cV$.

\subsection{Graded Morita theory}\label{subsec:morita}
The category $\cV$ is a $\cC$-bimodule category under $\otimes$. A left $\cC$-module functor $\cV\to\cV$ is a functor $G$ together with natural isomorphisms $s_{b,X}\colon b\otimes G(X)\to G(b\otimes X)$ ($b\in\cC$, $X\in\cV$), its \emph{module structure}. Our convention is inverse to that of \cite[Def.~7.2.1]{EGNO:2015}. We usually write $G$ for the module functor. The axioms are
\begin{subequations}\label{eq:module-functor}
\begin{align}
G(a_{b,b',X})s_{b\otimes b',X}
&=s_{b,b'\otimes X}(\id_b\otimes s_{b',X})a_{b,b',G(X)},\label{eq:module-assoc}\\
G(l_X)s_{\one,X}&=l_{G(X)}.\label{eq:module-unit}
\end{align}
\end{subequations}
Here both sides of the first equality have source
$(b\otimes b')\otimes G(X)$ and target $G(b\otimes(b'\otimes X))$
\cite[Def.~7.2.1, Rem.~7.2.5]{EGNO:2015}. A
natural transformation $\theta\colon G\Rightarrow G'$ to a module functor $G'$ with
module structure $s'$ is a \emph{module natural transformation} if the square
\begin{equation}\label{eq:module-nat}
\begin{tikzpicture}[baseline=(current bounding box.center)]
\node (tl) at (0,0) {$b\otimes G(X)$};
\node (tr) at (3.6,0) {$G(b\otimes X)$};
\node (bl) at (0,-1.6) {$b\otimes G'(X)$};
\node (br) at (3.6,-1.6) {$G'(b\otimes X)$};
\draw[->] (tl) -- node[above] {$s_{b,X}$} (tr);
\draw[->] (bl) -- node[below] {$s'_{b,X}$} (br);
\draw[->] (tl) -- node[left] {$\id_b\otimes\theta_X$} (bl);
\draw[->] (tr) -- node[right] {$\theta_{b\otimes X}$} (br);
\end{tikzpicture}
\end{equation}
commutes for all $b\in\cC$ and $X\in\cV$~\cite[Def.~7.2.2]{EGNO:2015}. For $a\in\cC$ let $R_a:=-\otimes a\colon\cV\to\cV$, $X\mapsto X\otimes a$, $f\mapsto f\otimes\id_a$, with module structure
\begin{equation*}
	s^{R_a}_{b,X}=a^{-1}_{b,X,a}\colon b\otimes(X\otimes a)\longrightarrow(b\otimes X)\otimes a.
\end{equation*}
The right unitor gives a module natural isomorphism $r\colon R_\one\Rightarrow\Id_\cV$.

\begin{fact}[Etingof--Nikshych--Ostrik]\label{fact:ENO}
Let $\cF=\bigoplus_{g\in G}\cF_g$ be a faithfully graded fusion category.
\begin{enumerate}
	\item Each $\cF_g$ is an invertible $\cF_e$-bimodule category, and tensor product induces equivalences of bimodule categories $\cF_g\boxtimes_{\cF_e}\cF_h\simeq\cF_{gh}$~\cite[Thm.~6.1]{Etingof:2009yvg}.
	\item If $\cM$ is an invertible $\cF_e$-bimodule category, then $a\mapsto R_a$ is an equivalence from $\cF_e^{\mathrm{rev}}$ to the category $\Fun_{\cF_e}(\cM,\cM)$ of left $\cF_e$-module endofunctors, with composition as tensor product~\cite[Prop.~4.2]{Etingof:2009yvg}.
\end{enumerate}
\end{fact}

We only use essential surjectivity and full faithfulness in Fact~\ref{fact:ENO}(2), so the conventions on ``rev'' and on the order of composition play no role.

\begin{corollary}\label{cor:morita}
Let $\cE=\cC\oplus\cV$ be as in Lemma~\ref{lem:grading}.
\begin{enumerate}
	\item Every left $\cC$-module endofunctor of $\cV$ is isomorphic, as a module functor, to $R_a$ for some $a\in\cC$.
	\item $\Hom_{\Fun_\cC(\cV,\cV)}(R_a,R_{a'})\cong\Hom_\cC(a,a')$. In particular $R_a\cong R_{a'}$ if and only if $a\cong a'$, and $R_\one\not\cong R_\psi$.
\end{enumerate}
\end{corollary}

\begin{proof}
The grading is faithful by Lemma~\ref{lem:grading}, so $\cV$ is an invertible $\cC$-bimodule category by Fact~\ref{fact:ENO}(1). By Fact~\ref{fact:ENO}(2), the functor $a\mapsto R_a$ from $\cC$ to $\Fun_\cC(\cV,\cV)$ is therefore essentially surjective and fully faithful.

(1) This is the essential surjectivity of $a\mapsto R_a$.

(2) The isomorphism is the full faithfulness of $a\mapsto R_a$. In particular, $R_a\cong R_{a'}$ if and only if $a\cong a'$. Since $\psi\not\cong\one$, we get $R_\one\not\cong R_\psi$.
\end{proof}

\section{Locality-respecting autoequivalences}\label{sec:LR}

In this section $\cB$ denotes a braided fusion category with a chosen fermion $\psi$, which may be transparent or nontransparent. We will apply the definitions to both $\cB=\cC$ and $\cB=\cE$. We fix an isomorphism $e\colon\psi\otimes\psi\to\one$, chosen unitary in the unitary setting.

The following definition specializes the fermionic tensor functors and natural transformations of~\cite[Defs.~3.5 and~3.9, Ex.~3.7]{Galindo-Venegas-Ramirez:2025} to braided autoequivalences, with the isomorphism $e$ kept explicit. The terminology \emph{locality-respecting} and the notation $\AutLR$ follow~\cite[\S IV.B]{Bulmash:2021hmb}.

\begin{definition}\label{def:LR}
A \emph{locality-respecting} (LR) autoequivalence of $(\cB,\psi)$ is a pair $(F,\varphi)$, where $F$ is a braided autoequivalence of $\cB$ and $\varphi\colon F(\psi)\to\psi$ is an isomorphism satisfying the \emph{local-vertex condition}~\cite[Def.~3.5, Eq.~(3.1)]{Galindo-Venegas-Ramirez:2025}
\begin{equation}\label{eq:vertex}
	e\circ(\varphi\otimes\varphi)=F(e)\circ J^F_{\psi,\psi}\colon F(\psi)\otimes F(\psi)\to\one .
\end{equation}
Here $F$ is unit-normalized.
This condition does not depend on the choice of $e$: replacing $e$ by $se$, $s\in\CC^\times$, multiplies both sides by $s$, because $F$ is linear. A \emph{morphism of pairs} $(F,\varphi)\to(G,\varphi')$ is a monoidal natural isomorphism $\eta\colon F\Rightarrow G$ with $\varphi'\circ\eta_\psi=\varphi$. Since $\eta$ is invertible, $\eta^{-1}$ is a morphism of pairs $(G,\varphi')\to(F,\varphi)$, so every morphism of pairs is an isomorphism. The \emph{composite} of pairs is
\begin{equation}\label{eq:compose-pairs}
	(F,\varphi)\cdot(G,\varphi'):=(F\circ G,\ \varphi\circ F(\varphi')).
\end{equation}
We write $[F,\varphi]$ for the isomorphism class of $(F,\varphi)$ and $\AutLR(\cB)$ for the set of classes. In the unitary setting, $F$, $\varphi$ and $\eta$ are required to be unitary, and the resulting set is
denoted $\AutLRu(\cB)$.
\end{definition}

\begin{remark}\label{rem:LR-meaning}
Since $\langle\psi\rangle\simeq\sVec$, condition~\eqref{eq:vertex} is precisely the condition that $\varphi$ extend uniquely to a monoidal natural isomorphism
\begin{equation*}
	\widehat\varphi\colon F|_{\langle\psi\rangle}\Rightarrow \Id_{\langle\psi\rangle},\qquad\widehat\varphi_{\one}=\id_{\one},\quad\widehat\varphi_\psi=\varphi.
\end{equation*}
Indeed, on simple objects the only nontrivial instance of the tensor condition in~\eqref{eq:monoidal-nat} for $\widehat\varphi$ is at $(\psi,\psi)$. There $\widehat\varphi_{\psi\otimes\psi}=e^{-1}F(e)$ by naturality of $\widehat\varphi$ at $e$, so the condition reads $e^{-1}F(e)J^F_{\psi,\psi}=\varphi\otimes\varphi$, which is~\eqref{eq:vertex}. Thus an LR pair records a \emph{chosen} monoidal trivialization on the fermion subcategory.

Moreover, if $\eta\colon F\Rightarrow G$ is a monoidal natural isomorphism and $(G,\varphi')$ is an LR pair, then $(F,\varphi'\circ\eta_\psi)$ is an LR pair, and $\eta$ is an isomorphism between these pairs by Definition~\ref{def:LR}. Indeed, by~\eqref{eq:vertex} for $(G,\varphi')$, the two conditions in~\eqref{eq:monoidal-nat} for $\eta$, and naturality of $\eta$ at $e$, we have
\begin{equation*}
	e\circ(\varphi'\eta_\psi\otimes\varphi'\eta_\psi)=G(e)\,J^G_{\psi,\psi}\,(\eta_\psi\otimes\eta_\psi)=G(e)\,\eta_{\psi\otimes\psi}\,J^F_{\psi,\psi}=\eta_\one\,F(e)\,J^F_{\psi,\psi}=F(e)\,J^F_{\psi,\psi}.
\end{equation*}
\end{remark}

Next we show that $\AutLR(\cB)$ is a group under the composition~\eqref{eq:compose-pairs}. We also show that a braided autoequivalence $F$ with $F(\psi)\cong\psi$ admits exactly two isomorphisms $F(\psi)\to\psi$ satisfying~\eqref{eq:vertex}, which differ by a sign, and that changing this sign amounts to multiplication by the central element $[\Id_\cB,-\id_\psi]$.

\begin{lemma}\label{lem:LR-group}
\begin{enumerate}
\item If $F$ is a braided autoequivalence with $F(\psi)\cong\psi$, then there exists an isomorphism $\varphi_0\colon F(\psi)\to\psi$ satisfying~\eqref{eq:vertex}, and the only isomorphisms satisfying this condition are $\varphi_0$ and $-\varphi_0$. In the unitary setting, if $F$ and $e$ are unitary, then $\varphi_0$ is unitary.
\item Composition~\eqref{eq:compose-pairs} induces a group structure on $\AutLR(\cB)$.
\item Put $\check\Upsilon:=[\Id_\cB,-\id_\psi]$. Then
$[F,-\varphi]=[F,\varphi]\cdot\check\Upsilon=\check\Upsilon\cdot[F,\varphi]$,
and $\check\Upsilon^2=1$.
\end{enumerate}
\end{lemma}

\begin{proof}
(1) Fix an isomorphism $\omega\colon F(\psi)\to\psi$. Since $\End(\psi)=\CC$, every isomorphism $F(\psi)\to\psi$ is $s\omega$ for a unique $s\in\CC^\times$. Both sides of~\eqref{eq:vertex} for $\omega$ are isomorphisms $F(\psi)\otimes F(\psi)\to\one$, and $\Hom(F(\psi)\otimes F(\psi),\one)\cong\End(\one)=\CC$, so $F(e)\circ J^F_{\psi,\psi}=\lambda\,e\circ(\omega\otimes\omega)$ for a unique $\lambda\in\CC^\times$. Hence $s\omega$ satisfies~\eqref{eq:vertex} if and only if $s^2=\lambda$, and the solutions are $\pm\varphi_0$ with $\varphi_0:=\sqrt\lambda\,\omega$. For $\psi\otimes\psi=\one$ and $e=\id$ this is~\cite[Prop.~3.6(1),(3)]{Galindo-Venegas-Ramirez:2025}. If $F$ and $e$ are unitary and $\varphi$ is a solution, then $\varphi\otimes\varphi=e^{-1}F(e)J^F_{\psi,\psi}$ is unitary. Writing $\varphi^\dagger\varphi=c\,\id_{F(\psi)}$ with $c>0$, we get $c^2=1$, so $c=1$ and $\varphi$ is unitary.

(2) By Remark~\ref{rem:LR-meaning}, an LR pair is a braided autoequivalence equipped with a chosen monoidal trivialization on $\langle\psi\rangle$. Composing these trivializations gives
\begin{equation*}
	FG(\psi)\xrightarrow{\,F(\varphi')\,}F(\psi)\xrightarrow{\,\varphi\,}\psi.
\end{equation*}
Hence $(FG,\varphi F(\varphi'))$ is again an LR pair.

We next check that the product is independent of representatives. Suppose $\eta\colon(F,\varphi)\Rightarrow(\widetilde F,\widetilde\varphi)$ and $\theta\colon(G,\varphi')\Rightarrow(\widetilde G,\widetilde\varphi')$ are isomorphisms of pairs. Their horizontal composite
\begin{equation*}
	(\eta*\theta)_X=\eta_{\widetilde G(X)}\circ F(\theta_X)\colon FG(X)\longrightarrow\widetilde F\widetilde G(X)
\end{equation*}
is a monoidal natural isomorphism. By naturality of $\eta$ and the marking conditions for $\eta,\theta$, we have	
\begin{equation*}
	\widetilde\varphi\,\widetilde F(\widetilde\varphi')\,(\eta*\theta)_\psi=\widetilde\varphi\,\eta_\psi\,
  F(\widetilde\varphi'\theta_\psi)=\varphi F(\varphi').
\end{equation*}
Thus it is an isomorphism of the composite pairs, so the product
descends to classes.

Composition is associative since both bracketings of three pairs have tensor functor $FGH$ and marking $\varphi F(\varphi')FG(\varphi'')$. The identity is $[\Id,\id_\psi]$.

It remains to construct inverses. We choose a unit-normalized tensor quasi-inverse $G$ of $F$ and a monoidal isomorphism $\epsilon\colon FG\Rightarrow\Id$. Since $FG\cong\Id$ monoidally, $FG$ is braided, and since $F$ is braided and faithful, $G$ is braided. Moreover, $G(\psi)\cong GF(\psi)\cong\psi$ by $\varphi$ and $GF\cong\Id$. By (1), choose a marking $\varphi'$ making $(G,\varphi')$ an LR pair. Both $\epsilon_\psi$ and $\varphi F(\varphi')$ are valid markings for $FG$, by Remark~\ref{rem:LR-meaning} and closure under composition. By (1), they differ by at most a sign. Replacing $\varphi'$ by $-\varphi'$ if necessary therefore makes $\epsilon$ an isomorphism
\begin{equation*}
	(F,\varphi)(G,\varphi')\cong(\Id,\id_\psi).
\end{equation*}
Thus every class has a right inverse. A monoid in which every element has a right inverse is a group. Indeed, if $ab=1$, choose $c$ with $bc=1$. Then $a=a(bc)=(ab)c=c$, so $ba=bc=1$.

(3) By~\eqref{eq:compose-pairs}, $(F,\varphi)(\Id,-\id_\psi)=(F,-\varphi)=(\Id,-\id_\psi)(F,\varphi)$, since $F(-\id_\psi)=-\id_{F(\psi)}$. For $(F,\varphi)=(\Id,-\id_\psi)$, we get $\check\Upsilon^2=[\Id,\id_\psi]=1$.
\end{proof}

\begin{definition}\label{def:restriction}
Let $\cC\subset\cE$ be a minimal nondegenerate extension. By Lemma~\ref{lem:grading}, every braided autoequivalence $F$ of $\cE$ with $F(\psi)\cong\psi$ preserves $\cC$, and so does a quasi-inverse, which also fixes $\psi$ up to isomorphism. Hence $F$ restricts to a braided autoequivalence $F|_\cC$ of $\cC$. The \emph{restriction homomorphism} is
\begin{equation*}
	r\colon\AutLR(\cE)\to\AutLR(\cC),\qquad [F,\varphi]\mapsto[F|_\cC,\varphi].
\end{equation*}
Thus $[F,\varphi]\in\ker r$ if and only if there is a morphism of pairs $(F|_\cC,\varphi)\to(\Id_\cC,\id_\psi)$, that is, a monoidal natural isomorphism $\eta\colon F|_\cC\Rightarrow\Id_\cC$ with $\eta_\psi=\varphi$. The same definitions apply to $\AutLRu$.
\end{definition}

\begin{lemma}\label{lem:unitary}
Let $\cB$ be a unitary braided fusion category with fermion $\psi$, and let $e$ be unitary. Then the forgetful map $\AutLRu(\cB)\to\AutLR(\cB)$ is bijective, and $\AutLRu(\cB)$ is a group under~\eqref{eq:compose-pairs}. If $\cC\subset\cE$ is a unitary minimal nondegenerate extension, then the bijections for $\cB=\cE$ and $\cB=\cC$ commute with the restriction maps, and hence identify the unitary and linear restriction homomorphisms.
\end{lemma}

\begin{proof}
Let $(F,\varphi)$ be an LR autoequivalence of $\cB$. By~\cite[Thm.~5]{Reutter:2023}, there is a monoidal natural isomorphism $\eta\colon F\Rightarrow F'$ onto a unitary monoidal autoequivalence $F'$. It is braided, because $J^{F'}$ is transported from $J^F$ along $\eta$ (Section~\ref{subsec:tensor-functors}). We write $\varphi':=\varphi\circ\eta_\psi^{-1}$. By Remark~\ref{rem:LR-meaning}, $(F',\varphi')$ is an LR autoequivalence and $\eta$ is a morphism of pairs $(F,\varphi)\to(F',\varphi')$. By Lemma~\ref{lem:LR-group}(1), $\varphi'$ is unitary. Hence the forgetful map is surjective. A morphism of pairs between unitary pairs is a monoidal natural isomorphism between unitary monoidal equivalences, hence unitary by~\cite[Thm.~2]{Reutter:2023}. Hence the forgetful map is injective. It is multiplicative, since composites of unitary pairs are unitary, so $\AutLRu(\cB)$ is a group. Restriction to $\cC$ maps unitary pairs to unitary pairs, so the bijections commute with $r$.
\end{proof}

\begin{remark}\label{rem:BB}
In~\cite{Barkeshli:2014cna,Bulmash:2021hmb,Bulmash:2021ryq} a symmetry is described in a skeletal gauge. It consists of a permutation $\rho$ of the simple labels and $U$-symbols $U_\rho(a,b;c)$, the matrices in chosen bases of the fusion-space maps $V^{ab}_c\to V^{\rho(a)\rho(b)}_{\rho(c)}$, $v\mapsto\rho(v)\circ J^\rho_{a,b}$, where $V^{ab}_c=\Hom(a\otimes b,c)$.

A natural isomorphism between two symmetries with the same permutation is, by Schur's lemma~\cite[Prop.~1.8.4]{EGNO:2015}, a family of scalars $\gamma_a$. It changes the $U$-symbols by the ratio of the $\gamma$'s~\cite[Eq.~(22)]{Bulmash:2021ryq}.

Bulmash and Barkeshli require $U_\rho(\psi,\psi;\one)=1$~\cite[Eq.~(65)]{Bulmash:2021ryq}, and they allow only natural isomorphisms admitting a presentation with $\gamma_\psi=1$~\cite[Eq.~(77)]{Bulmash:2021ryq}. In a skeletal model $F(\psi)=\psi$, and the dictionary is as follows:
\begin{itemize}
\item $U_F(\psi,\psi;\one)=1$ is \eqref{eq:vertex} with $\varphi=\id_\psi$;
\item a natural isomorphism with $\gamma_\psi=1$ is a morphism of pairs
$(F,\id_\psi)\to(G,\id_\psi)$.
\end{itemize}
More generally, we write $\varphi=t\,\id_\psi$ with $t\in\CC^\times$ in a skeletal model and then transport $F$ along the natural isomorphism whose component at $\psi$ is $\varphi$ and whose other simple components are identities. By Lemma~\ref{lem:transport} and Remark~\ref{rem:LR-meaning}, we obtain an isomorphic LR pair $(F^\gamma,\id_\psi)$, for which $U_{F^\gamma}(\psi,\psi;\one)=1$. Thus the dictionary covers every LR pair. In particular, every LR pair $(F,-\id_\psi)$ is isomorphic to a pair $(F^\gamma,\id_\psi)$. Here $F^\gamma$ is $F$ transported along the natural isomorphism with $\gamma_\psi=-1$ and $\gamma_a=1$ for $a\neq\psi$. Accordingly, the natural isomorphisms $\Upsilon_\psi$ and $\check\Upsilon_\psi$ of \cite[Eqs.~(72)--(73)]{Bulmash:2021ryq} correspond to the classes $[\Id_\cC,-\id_\psi]$ and $[\Id_\cE,-\id_\psi]$.

Consequently $\AutLRu(\cB)$ is the group denoted $\operatorname{Aut}_{LR}(\cB)$ in~\cite[\S IV.C]{Bulmash:2021ryq} and the fermionic topological symmetry group $\operatorname{Aut}_f(\cB)$ of~\cite[\S III.A]{Aasen:2021vva}, which is defined by the same two conditions~\cite[Eqs.~(104) and~(107)]{Aasen:2021vva}, in both cases restricted to complex-linear symmetries. The map $r$ is the restriction map of~\cite[Eq.~(83)]{Bulmash:2021ryq}.
\end{remark}

\section{The vortex sector of a kernel element}\label{sec:vortex}

In this section $\cC\subset\cE$ is a minimal nondegenerate extension, and we write $\cE=\cC\oplus\cV$ for the parity grading of Lemma~\ref{lem:grading}.

\begin{definition}\label{def:normalized}
A pair $(\Phi,\id_\psi)$ is \emph{normalized} if $\Phi$ is a braided autoequivalence of $\cE$ whose restriction to $\cC$ is the identity tensor functor. That is, $\Phi(b)=b$, $\Phi(f)=f$, $J^\Phi_{b,b'}=\id$ and $\iota^\Phi=\id_\one$ for all objects $b,b'$ and morphisms $f$ of $\cC$. The local-vertex condition~\eqref{eq:vertex} with $\varphi=\id_\psi$ then holds automatically, since $e$ is a morphism of $\cC$ and hence $\Phi(e)\circ J^\Phi_{\psi,\psi}=e=e\circ(\id_\psi\otimes\id_\psi)$.
\end{definition}

\begin{lemma}\label{lem:normalization}
Every class in $\ker r$ is represented by a normalized pair. If $(\Phi,\id_\psi)$ is normalized, then $\Phi(\cV)\subseteq\cV$, and $\Phi|_\cV$ with the module structure
\begin{equation*}
	J^\Phi_{b,X}\colon b\otimes\Phi(X)\to\Phi(b\otimes X)\qquad(b\in\cC,\ X\in\cV),
\end{equation*}
is a left $\cC$-module autoequivalence of $\cV$.
\end{lemma}

\begin{proof}
Let $[F,\varphi]\in\ker r$. We construct a normalized pair $(\Phi,\id_\psi)$ with $[\Phi,\id_\psi]=[F,\varphi]$. By Definition~\ref{def:restriction}, there is a monoidal natural isomorphism $\eta\colon F|_\cC\Rightarrow\Id_\cC$ with $\eta_\psi=\varphi$. We define a functor $\Phi\colon\cE\to\cE$ and a natural isomorphism $\Theta\colon F\Rightarrow\Phi$ on homogeneous objects (Convention~\ref{conv:graded}) by
\begin{equation*}
	\Phi(b)=b,\quad\Theta_b=\eta_b\quad(b\in\cC),\qquad\Phi(X)=F(X),\quad\Theta_X=\id_{F(X)}\quad(X\in\cV),
\end{equation*}
and on a morphism $f\colon U\to U'$ between homogeneous objects $U,U'$ of $\cE$ by $\Phi(f)=\Theta_{U'}\,F(f)\,\Theta_U^{-1}$, so that $\Theta$ is natural by construction. Here $\Phi(f)$ is the unique morphism with $\Theta_{U'}\,F(f)=\Phi(f)\,\Theta_U$, by~\cite[Lem.~1.5.10]{Riehl:2016} applied to $F(f)$ and the isomorphisms $\Theta_U$ and $\Theta_{U'}$. By this uniqueness, $\Phi(gf)=\Phi(g)\Phi(f)$ for $g\colon U'\to U''$ and $\Phi(\id_U)=\id_{\Phi(U)}$, so $\Phi$ is a functor. By~\eqref{eq:transport}, $\Phi$ has a unique braided tensor structure $J^\Phi$ for which $\Theta$ is monoidal, namely the one for which the square
\begin{equation*}
\begin{tikzpicture}[baseline=(current bounding box.center)]
\node (tl) at (0,0) {$F(X)\otimes F(Y)$};
\node (tr) at (4.4,0) {$F(X\otimes Y)$};
\node (bl) at (0,-1.6) {$\Phi(X)\otimes\Phi(Y)$};
\node (br) at (4.4,-1.6) {$\Phi(X\otimes Y)$};
\draw[->] (tl) -- node[above] {$J^F_{X,Y}$} (tr);
\draw[->] (bl) -- node[below] {$J^\Phi_{X,Y}$} (br);
\draw[->] (tl) -- node[left] {$\Theta_X\otimes\Theta_Y$} (bl);
\draw[->] (tr) -- node[right] {$\Theta_{X\otimes Y}$} (br);
\end{tikzpicture}
\end{equation*}
commutes for all $X,Y\in\cE$. Since $\Theta$ is an isomorphism and $F$ is an autoequivalence, so is $\Phi$. For $b,b'\in\cC$ the vertical arrows are $\eta_b\otimes\eta_{b'}$ and $\eta_{b\otimes b'}$, and by~\eqref{eq:monoidal-nat} for $\eta$ the square also commutes with $\id_{b\otimes b'}$ in place of $J^\Phi_{b,b'}$. Hence $J^\Phi_{b,b'}=\id$. Moreover, $\iota^\Phi=\Theta_\one\iota^F=\eta_\one\iota^F=\id_\one$ by~\eqref{eq:transport} and the unit condition in~\eqref{eq:monoidal-nat}. For a morphism $f\colon b\to b'$ of $\cC$, $\Phi(f)=\eta_{b'}\,F(f)\,\eta_b^{-1}=f$ by naturality of $\eta$. Thus $\Phi|_\cC=\Id_\cC$ as a tensor functor, and $(\Phi,\id_\psi)$ is normalized. Since $\id_\psi\circ\Theta_\psi=\varphi$, the isomorphism $\Theta$ is a morphism of pairs $(F,\varphi)\to(\Phi,\id_\psi)$, so $[F,\varphi]=[\Phi,\id_\psi]$.

Now let $(\Phi,\id_\psi)$ be normalized. By Lemma~\ref{lem:grading}, $\Phi(\cV)\subseteq\cV$. Hence $\Phi$ is an equivalence of braided $\Z/2$-extensions of $\cC$ in the sense of~\cite[Defs.~8.2, 8.16]{Davydov-Nikshych:2021}, and by~\cite[\S8.1, Eq.~(8.5)]{Davydov-Nikshych:2021} it restricts to a $\cC$-bimodule autoequivalence of $\cV$. In particular $\Phi|_\cV$ is a left $\cC$-module autoequivalence. Its module structure is $s_{b,X}:=J^\Phi_{b,X}$ for $b\in\cC$ and $X\in\cV$. Indeed, since $\Phi(b)=b$ and $J^\Phi_{b,b'}=\id$ for $b,b'\in\cC$, the monoidal axiom~\eqref{eq:monoidal-axiom} for $\Phi$ at $(b,b',X)$ is~\eqref{eq:module-assoc}, while the left unit condition in~\eqref{eq:functor-units} for $\Phi$ at $X$ is~\eqref{eq:module-unit} since $\iota^\Phi=\id_\one$. Hence $\Phi|_\cV$, with the module structure $s$, is a left $\cC$-module autoequivalence of $\cV$.
\end{proof}

\subsection{The module type}\label{subsec:module-type}

By Lemma~\ref{lem:normalization}, every element of $\ker r$ is represented by a normalized pair $(\Phi,\id_\psi)$, so all of its information sits in the vortex sector $\cV$. We first determine $\Phi|_\cV$ as a module functor. By graded Morita theory (Corollary~\ref{cor:morita}), every left $\cC$-module endofunctor of $\cV$ is isomorphic to $R_a=-\otimes a$ for some $a\in\cC$. So $\Phi(X)\cong X\otimes a$ for every vortex $X$, compatibly with fusion with $\cC$. Since $\Phi$ is moreover braided and is the identity on $\cC$, fusing with $a$ must not change the double braiding of a vortex with any $b\in\cC$. So $a$ must be transparent in $\cC$, hence a sum of copies of $\one$ and $\psi$, and $a$ is simple because $\Phi$ is an equivalence. So there are only two module types, namely $\Phi$ either fixes the vortices, or attaches the fermion to them. Thus we have the following statement.

\begin{proposition}\label{prop:module-type}
Let $(\Phi,\id_\psi)$ be normalized. There is a unique $a\in\{\one,\psi\}$ such that $\Phi|_\cV\cong R_a$ as left $\cC$-module functors, where $\Phi|_\cV$ carries the module structure of Lemma~\ref{lem:normalization}.
\end{proposition}

\begin{proof}
By Lemma~\ref{lem:normalization} and Corollary~\ref{cor:morita}(1), there are $a\in\cC$
and a module natural isomorphism $t\colon\Phi|_\cV\Rightarrow R_a$. Using the module structure $s^{R_a}_{b,X}=a^{-1}_{b,X,a}$,~\eqref{eq:module-nat} reads
\begin{equation}\label{eq:t-module}
	t_{b\otimes X}\circ J^\Phi_{b,X}=a^{-1}_{b,X,a}(\id_b\otimes t_X)\qquad(b\in\cC,\ X\in\cV).
\end{equation}

We show that $a\in\Zt(\cC)$. For $b\in\cC$ and an object $Y$ put $D_{b,Y}:=c_{Y,b}\circ c_{b,Y}\in\End(b\otimes Y)$. Let $X\in\cV$, so that $b\otimes X\in\cV$.
We compare the monodromy of $b$ with $X\otimes a$ with that of $b$ with $X$ by considering the diagram
\begin{equation*}
\begin{tikzpicture}[baseline=(current bounding box.center)]
\node (a1) at (0,0) {$b\otimes\Phi(X)$};
\node (a2) at (3.4,0) {$\Phi(b\otimes X)$};
\node (a3) at (6.8,0) {$(b\otimes X)\otimes a$};
\node (b1) at (0,-1.6) {$\Phi(X)\otimes b$};
\node (b2) at (3.4,-1.6) {$\Phi(X\otimes b)$};
\node (c1) at (0,-3.2) {$b\otimes\Phi(X)$};
\node (c2) at (3.4,-3.2) {$\Phi(b\otimes X)$};
\node (c3) at (6.8,-3.2) {$(b\otimes X)\otimes a$};
\draw[->] (a1) -- node[above] {$J^\Phi_{b,X}$} (a2);
\draw[->] (a2) -- node[above] {$t_{b\otimes X}$} (a3);
\draw[->] (b1) -- node[above] {$J^\Phi_{X,b}$} (b2);
\draw[->] (c1) -- node[below] {$J^\Phi_{b,X}$} (c2);
\draw[->] (c2) -- node[below] {$t_{b\otimes X}$} (c3);
\draw[->] (a1) -- node[left] {$c_{b,\Phi(X)}$} (b1);
\draw[->] (b1) -- node[left] {$c_{\Phi(X),b}$} (c1);
\draw[->] (a2) -- node[right] {$\Phi(c_{b,X})$} (b2);
\draw[->] (b2) -- node[right] {$\Phi(c_{X,b})$} (c2);
\draw[->] (a3) -- node[right] {$D_{b,X}\otimes\id_a$} (c3);
\end{tikzpicture}
\end{equation*}
Its left and middle columns compose to $D_{b,\Phi(X)}$ and $\Phi(D_{b,X})$. Since $\Phi(b)=b$, the two squares on the left are~\eqref{eq:braided-functor} for $\Phi$ at $(b,X)$ and at $(X,b)$. The rectangle on the right is the naturality of $t$ at $D_{b,X}$, since $R_a(D_{b,X})=D_{b,X}\otimes\id_a$. By~\eqref{eq:t-module}, the top and bottom rows both compose to the isomorphism $T:=a^{-1}_{b,X,a}(\id_b\otimes t_X)\colon b\otimes\Phi(X)\to(b\otimes X)\otimes a$. Since the outer rectangle commutes, $D_{b,X}\otimes\id_a=T\,D_{b,\Phi(X)}\,T^{-1}$. By naturality of the monodromy at $t_X$, we have $(\id_b\otimes t_X)\,D_{b,\Phi(X)}\,(\id_b\otimes t_X)^{-1}=D_{b,X\otimes a}$, and therefore
\begin{equation}\label{eq:DXa}
a^{-1}_{b,X,a}D_{b,X\otimes a}a_{b,X,a}=D_{b,X}\otimes\id_a.
\end{equation}
Using the hexagon axioms we can express the left-hand side as
\begin{equation*}
	(c_{X,b}\otimes\id_a)a^{-1}_{X,b,a}(\id_X\otimes D_{b,a})a_{X,b,a}(c_{b,X}\otimes\id_a)
\end{equation*}
whereas the right-hand side of~\eqref{eq:DXa} is $(c_{X,b}\otimes\id_a)(c_{b,X}\otimes\id_a)$. Cancelling the isomorphisms $c_{X,b}\otimes\id_a$ on the left and $c_{b,X}\otimes\id_a$ on the right, we get $a^{-1}_{X,b,a}(\id_X\otimes D_{b,a})a_{X,b,a}=\id_{(X\otimes b)\otimes a}$, and hence $\id_X\otimes D_{b,a}=\id_{X\otimes(b\otimes a)}$. Choosing $X\neq0$, which is possible since $\cV\neq0$, we get $D_{b,a}=\id_{b\otimes a}$ for all $b\in\cC$ by Lemma~\ref{lem:faithful}(1). Thus $a\in\Zt(\cC)=\langle\psi\rangle$, that is, $a\cong\one^{\oplus m}\oplus\psi^{\oplus n}$ for some $m,n\ge0$.

Let $X\in\cV$ be simple. Since $\Phi$ is fully faithful, we have $\End(X\otimes a)\cong\End(\Phi(X))\cong\End(X)=\CC$, where the first isomorphism is conjugation by $t_X$. On the other hand $X\otimes a\cong X^{\oplus m}\oplus(X\otimes\psi)^{\oplus n}$ is nonzero, and $\dim\End(X\otimes a)\ge m^2+n^2$. Hence $(m,n)\in\{(1,0),(0,1)\}$, that is, $a\cong\one$ or $a\cong\psi$. Finally, uniqueness of $a$ follows from Corollary~\ref{cor:morita}(2).
\end{proof}

\begin{remark}\label{rem:twist-argument}
If $\cE$ is pseudo-unitary, the transparency of $a$ in the proof of Proposition~\ref{prop:module-type} can also be seen using twists, that is, topological spins. First, $a$ is invertible. Indeed, $R_a\cong\Phi|_\cV$ is an autoequivalence of $\cV$, so it permutes the simple objects of $\cV$. Summing over them, we get $\sum_X\FPdim(X)^2=\sum_X\FPdim(X\otimes a)^2=\FPdim(a)^2\sum_X\FPdim(X)^2$, so $\FPdim(a)=1$ and $a$ is invertible~\cite[Exercise~4.5.9]{EGNO:2015}. For a simple object $Y$ of $\cE$, let $m_a(Y)$ be the monodromy scalar of $Y$ with $a$, that is, $c_{a,Y}\circ c_{Y,a}=m_a(Y)\,\id_{Y\otimes a}$.

Since $\cE$ is pseudo-unitary, it has a unique spherical structure with $d_Y=\FPdim(Y)$ for simple $Y$~\cite[Prop.~9.5.1]{EGNO:2015}, and we let $\theta$ be the corresponding ribbon structure~\cite[(8.35), Prop.~8.10.12]{EGNO:2015}. It satisfies the balancing identity~\cite[Def.~8.10.1]{EGNO:2015}
\begin{equation}\label{eq:balancing}
\theta_{U\otimes V}=(\theta_U\otimes\theta_V)\circ c_{V,U}\circ c_{U,V}\qquad(U,V\in\cE).
\end{equation}
For simple $Y$ we write $\theta_Y$ also for the scalar by which $\theta$ acts on $Y$. A braided autoequivalence $F$ of $\cE$ preserves $\theta$. Indeed, it transports the spherical structure to one with the same dimensions $\FPdim(F(Y))=\FPdim(Y)$~\cite[Prop.~4.5.7]{EGNO:2015}, hence to itself by this uniqueness, and it preserves the Drinfeld isomorphism $u$~\cite[(8.30)]{EGNO:2015}, which is built from the braiding and the duality. Hence $\theta_{F(Y)}=\theta_Y$ for simple $Y$.

Let $X$ be a simple vortex. Since $\Phi(X)\cong X\otimes a$, we get $\theta_{X\otimes a}=\theta_X$. Since $a$ is invertible, $X\otimes a$ is simple (Section~\ref{subsec:fermions}), so all terms of the balancing identity at $(U,V)=(X,a)$ are scalar multiples of $\id_{X\otimes a}$, and it reads $\theta_{X\otimes a}=\theta_X\,\theta_a\,m_a(X)$. Hence $m_a(X)=\theta_a^{-1}$ for every simple vortex $X$.

Now let $b\in\cC$ be simple. Since $b\otimes X\neq0$ by Lemma~\ref{lem:faithful}(1), it has a simple summand $Y$, and $Y$ is a vortex because $\cC\otimes\cV\subseteq\cV$. By~\cite[Lemma~8.22.1(i)]{EGNO:2015} and naturality of the braiding, as in the proof of Lemma~\ref{lem:grading}, $m_a(Y)=m_a(b)\,m_a(X)$. Since $X$ and $Y$ are simple vortices, $m_a(X)=m_a(Y)=\theta_a^{-1}$, so $\theta_a^{-1}=m_a(b)\,\theta_a^{-1}$ and $m_a(b)=1$. Thus every simple object of $\cC$ has trivial monodromy with $a$, that is, $a\in\Zt(\cC)$. We note that the proof of Proposition~\ref{prop:module-type} is purely braided and needs no ribbon structure.
\end{remark}

\subsection{Relative rigidity}\label{subsec:rigidity}

By Proposition~\ref{prop:module-type}, a normalized pair $(\Phi,\id_\psi)$ is determined on $\cC$, and on $\cV$ up to isomorphism of module functors. However, its class is not yet determined by these data. The tensor structure of $\Phi$ also has components $J^\Phi_{X,Y}\colon\Phi(X)\otimes\Phi(Y)\to\Phi(X\otimes Y)$ for $X,Y\in\cV$, which describe how $\Phi$ fuses two vortices into an object of $\cC$, and they are not determined directly by the restriction to $\cC$ and the module type $R_a$. We show in the following lemma that these components carry no further information. Namely, if the module type is $R_\one$, which is isomorphic to $\Id_\cV$ as a module functor by the right unitor $r$ (Section~\ref{subsec:morita}), then $\Phi\cong\Id_\cE$ monoidally by an isomorphism that is the identity on $\cC$, hence $[\Phi,\id_\psi]=[\Id_\cE,\id_\psi]$ is the identity of $\AutLR(\cE)$. In the proof, the odd--odd components are reduced to a single scalar $\lambda$, which is removed by rescaling the odd part by a square root of $\lambda$. 

\begin{lemma}[Relative rigidity]\label{lem:rigidity}
Let $\cF=\cA\oplus\cM$ be a braided fusion category with a faithful $\Z/2$-grading with neutral component $\cA$. Let $F\colon\cF\to\cF$ be a braided tensor functor with tensor structure $J$ and the following properties:
\begin{enumerate}
\item $F|_\cA=\Id_\cA$ as a tensor functor;
\item $F(\cM)\subseteq\cM$;
\item $F|_\cM$, with the module structure $J|_{\cA\times\cM}$, is isomorphic to $\Id_\cM$ as a left $\cA$-module functor.
\end{enumerate}
Then there is a monoidal natural isomorphism $\xi\colon F\Rightarrow\Id_\cF$ with $\xi_a=\id_a$ for all $a\in\cA$.
\end{lemma}

\begin{proof}
We choose a module natural isomorphism $\theta\colon F|_\cM\Rightarrow\Id_\cM$, and let $F'$ be the identity functor of $\cF$ with the tensor structure $J'$ transported from $J$ along $\Theta:=\id_\cA\oplus\theta\colon F\Rightarrow F'$ as in~\eqref{eq:transport}. Then $J'$ is braided and $\Theta$ is monoidal. Since $\Theta_a=\id_a$ for $a\in\cA$, it suffices to construct a monoidal natural isomorphism $\gamma\colon F'\Rightarrow\Id_\cF$ with $\gamma_a=\id_a$ for $a\in\cA$, since then $\xi:=\gamma\circ\Theta$ is a monoidal natural isomorphism $F\Rightarrow\Id_\cF$ with $\xi_a=\id_a$.

By~\eqref{eq:transport}, $J'$ is determined by the commutative squares
\begin{equation*}
\begin{tikzpicture}[baseline=(current bounding box.center)]
\node (tl) at (0,0) {$F(U)\otimes F(V)$};
\node (tr) at (4.4,0) {$F(U\otimes V)$};
\node (bl) at (0,-1.6) {$U\otimes V$};
\node (br) at (4.4,-1.6) {$U\otimes V$};
\draw[->] (tl) -- node[above] {$J_{U,V}$} (tr);
\draw[->] (bl) -- node[below] {$J'_{U,V}$} (br);
\draw[->] (tl) -- node[left] {$\Theta_U\otimes\Theta_V$} (bl);
\draw[->] (tr) -- node[right] {$\Theta_{U\otimes V}$} (br);
\end{tikzpicture}
\end{equation*}
for $U,V\in\cF$. For $a,a'\in\cA$, the arrows $J_{a,a'}$, $\Theta_a\otimes\Theta_{a'}$ and $\Theta_{a\otimes a'}$ are identities by (1) and the fact that $\Theta_a=\id_a$. For $a\in\cA$ and $X\in\cM$, the vertical arrows are $\id_a\otimes\theta_X$ and $\theta_{a\otimes X}$, and by~\eqref{eq:module-nat} for $\theta$ the square also commutes with
$\id_{a\otimes X}$ in place of $J'_{a,X}$. Hence
\begin{equation}\label{eq:J-trivial-1}
	J'_{a,a'}=\id,\qquad J'_{a,X}=\id\qquad \forall a,a'\in\cA,\ \forall X\in\cM.
\end{equation}
By~\eqref{eq:braided-functor} for $F'$ at $(a,X)$, $c_{a,X}J'_{a,X}=J'_{X,a}c_{a,X}$, so
\begin{equation}\label{eq:J-trivial-2}
	J'_{X,a}=\id\qquad \forall a\in\cA,\ \forall X\in\cM.
\end{equation}
The remaining components $J'_{X,Y}$ with $X,Y\in\cM$ are automorphisms of $X\otimes Y\in\cA$. By~\eqref{eq:J-trivial-1} and~\eqref{eq:J-trivial-2}, the monoidal axiom~\eqref{eq:monoidal-axiom} for $F'$ at $(a,X,Y)$, $(X,a,Y)$ and $(X,Y,a)$, with $a\in\cA$ and $X,Y\in\cM$, reads
\begin{equation}\label{eq:LBR}
\begin{aligned}
	a_{a,X,Y}J'_{a\otimes X,Y}&=(\id_a\otimes J'_{X,Y})a_{a,X,Y},\\
	a_{X,a,Y}J'_{X\otimes a,Y}&=J'_{X,a\otimes Y}a_{X,a,Y},\\
	J'_{X,Y\otimes a}a_{X,Y,a}&=a_{X,Y,a}(J'_{X,Y}\otimes\id_a).
\end{aligned}
\end{equation}
Let $X,Y,Z,W\in\cM$ be nonzero, and consider the diagram
\begin{equation*}
\begin{tikzpicture}[baseline=(current bounding box.center),every node/.style={font=\small}]
\node (a1) at (0,0) {$((X\otimes Y)\otimes Z)\otimes W$};
\node (b1) at (6,0) {$((X\otimes Y)\otimes Z)\otimes W$};
\node (a2) at (0,-1.9) {$(X\otimes(Y\otimes Z))\otimes W$};
\node (b2) at (6,-1.9) {$(X\otimes(Y\otimes Z))\otimes W$};
\node (a3) at (0,-3.8) {$X\otimes((Y\otimes Z)\otimes W)$};
\node (b3) at (6,-3.8) {$X\otimes((Y\otimes Z)\otimes W)$};
\node (a4) at (0,-5.7) {$X\otimes(Y\otimes(Z\otimes W))$};
\node (b4) at (6,-5.7) {$X\otimes(Y\otimes(Z\otimes W))$};
\node (a5) at (0,-7.6) {$(X\otimes Y)\otimes(Z\otimes W)$};
\node (b5) at (6,-7.6) {$(X\otimes Y)\otimes(Z\otimes W)$};
\draw[->] (a1) -- node[above] {$J'_{(X\otimes Y)\otimes Z,W}$} (b1);
\draw[->] (a2) -- node[above] {$J'_{X\otimes(Y\otimes Z),W}$} (b2);
\draw[->] (a3) -- node[above] {$J'_{X,(Y\otimes Z)\otimes W}$} (b3);
\draw[->] (a4) -- node[above] {$J'_{X,Y\otimes(Z\otimes W)}$} (b4);
\draw[->] (a5) -- node[above] {$J'_{X,Y}\otimes\id_{Z\otimes W}$} (b5);
\draw[->] (a1) -- node[left] {$a_{X,Y,Z}\otimes\id_W$} (a2);
\draw[->] (a2) -- node[left] {$a_{X,Y\otimes Z,W}$} (a3);
\draw[->] (a3) -- node[left] {$\id_X\otimes a_{Y,Z,W}$} (a4);
\draw[->] (a4) -- node[left] {$a^{-1}_{X,Y,Z\otimes W}$} (a5);
\draw[->] (b1) -- node[right] {$a_{X,Y,Z}\otimes\id_W$} (b2);
\draw[->] (b2) -- node[right] {$a_{X,Y\otimes Z,W}$} (b3);
\draw[->] (b3) -- node[right] {$\id_X\otimes a_{Y,Z,W}$} (b4);
\draw[->] (b4) -- node[right] {$a^{-1}_{X,Y,Z\otimes W}$} (b5);
\node[font=\scriptsize] at (3,-0.72) {naturality of $J'$};
\node[font=\scriptsize] at (3,-2.62) {\eqref{eq:LBR} at $(X,Y\otimes Z,W)$};
\node[font=\scriptsize] at (3,-4.52) {naturality of $J'$};
\node[font=\scriptsize] at (3,-6.42) {\eqref{eq:LBR} at $(X,Y,Z\otimes W)$};
\end{tikzpicture}
\end{equation*}
whose cells commute for the reasons indicated in the middle of the corresponding cells. By the pentagon axiom, each vertical composite equals $a_{X\otimes Y,Z,W}$. Hence $(J'_{X,Y}\otimes\id_{Z\otimes W})\,a_{X\otimes Y,Z,W}=a_{X\otimes Y,Z,W}\,J'_{(X\otimes Y)\otimes Z,W}$, and the right-hand side equals $(\id_{X\otimes Y}\otimes J'_{Z,W})\,a_{X\otimes Y,Z,W}$ by~\eqref{eq:LBR} at $(X\otimes Y,Z,W)$. Cancelling $a_{X\otimes Y,Z,W}$, we get
\begin{equation}\label{eq:four-odd}
	\id_{X\otimes Y}\otimes J'_{Z,W}=J'_{X,Y}\otimes\id_{Z\otimes W}\in \End((X\otimes Y)\otimes(Z\otimes W)).
\end{equation}
The objects $X\otimes Y$ and $Z\otimes W$ are nonzero, since $\id_X\otimes\id_Y\neq0$ and $\id_Z\otimes\id_W\neq0$ by Lemma~\ref{lem:faithful}(1). So, by Lemma~\ref{lem:faithful}(2), there is $\lambda\in\CC^\times$ with $J'_{X,Y}=\lambda\,\id$ and $J'_{Z,W}=\lambda\,\id$. Hence $\lambda$ does not depend on the pair, and
\begin{equation}\label{eq:J-scalar}
	J'_{X,Y}=\lambda\,\id_{X\otimes Y}\qquad \forall X,Y\in\cM.
\end{equation}

Choose $t\in\CC^\times$ with $t^2=\lambda$, and define $\gamma\colon F'\Rightarrow\Id_\cF$ by $\gamma_a=\id_a$ for $a\in\cA$ and $\gamma_X=t\,\id_X$ for $X\in\cM$. Since $\Id_\cF$ has trivial tensor structure, the tensor condition in~\eqref{eq:monoidal-nat} for $\gamma$ reads $\gamma_{U\otimes V}J'_{U,V}=\gamma_U\otimes\gamma_V$. Both sides are natural in $U$ and $V$, so it suffices to check it for homogeneous $U$ and $V$. By~\eqref{eq:J-trivial-1},~\eqref{eq:J-trivial-2} and~\eqref{eq:J-scalar}, both sides equal $\id_{U\otimes V}$ if $U,V\in\cA$, both equal $t\,\id_{U\otimes V}$ if exactly one of $U,V$ lies in $\cM$, and they equal $\lambda\,\id_{U\otimes V}$ and $t^2\,\id_{U\otimes V}=\lambda\,\id_{U\otimes V}$ if $U,V\in\cM$. By~\eqref{eq:transport} and (1), $\iota^{F'}=\Theta_\one\iota^F=\id_\one$, so the unit condition reads $\gamma_\one=\id_\one$, which holds since $\one\in\cA$. So $\gamma$ is monoidal, and $\xi:=\gamma\circ\Theta\colon F\Rightarrow\Id_\cF$ is a monoidal natural isomorphism with $\xi_a=\id_a$ for all $a\in\cA$.
\end{proof}

\begin{remark}\label{rem:rigidity-cocycle}
In the proof of Lemma~\ref{lem:rigidity}, the tensor structure $J'$ is trivialized in two steps, which we phrase in terms of the group cohomology of $\Z/2$ with coefficients in $\CC^\times$. For a homogeneous object $U\in\cF$, we write $p(U)\in\{0,1\}$ for its degree, and we add degrees modulo $2$. Since the grading is faithful, every degree is the degree of a nonzero object, and tensor products of nonzero objects are nonzero by Lemma~\ref{lem:faithful}(1).
\begin{enumerate}
\item By~\eqref{eq:J-trivial-1},~\eqref{eq:J-trivial-2} and~\eqref{eq:J-scalar}, we have $J'_{U,V}=\omega(p(U),p(V))\,\id_{U\otimes V}$ for homogeneous $U,V\in\cF$, where $\omega(g,h):=\lambda^{gh}$ for $g,h\in\{0,1\}$. A priori, the components $J'_{X,Y}$ with $X,Y\in\cM$ are arbitrary automorphisms of $X\otimes Y$, and the essential input is that, by~\eqref{eq:four-odd} and Lemma~\ref{lem:faithful}(2), they are all equal to $\lambda\,\id_{X\otimes Y}$ for a single $\lambda\in\CC^\times$. Let $U,V,W\in\cF$ be nonzero and homogeneous of degrees $g,h,k$. Since $F'$ is the identity functor, the monoidal axiom~\eqref{eq:monoidal-axiom} for $F'$ at $(U,V,W)$ reads
\begin{equation*}
	a_{U,V,W}\,J'_{U\otimes V,W}\,(J'_{U,V}\otimes\id_W)=J'_{U,V\otimes W}\,(\id_U\otimes J'_{V,W})\,a_{U,V,W}.
\end{equation*}
The objects $U\otimes V$ and $V\otimes W$ are homogeneous of degrees $g+h$ and $h+k$, so the left-hand side is $\omega(g+h,k)\,\omega(g,h)\,a_{U,V,W}$ and the right-hand side is $\omega(g,h+k)\,\omega(h,k)\,a_{U,V,W}$. Cancelling $a_{U,V,W}$ and using $(U\otimes V)\otimes W\neq0$, we get
\begin{equation*}
	\omega(h,k)\,\omega(g,h+k)=\omega(g+h,k)\,\omega(g,h).
\end{equation*}
By faithfulness of the grading, this identity holds for all $g,h,k\in\Z/2$, which is the $2$-cocycle condition for $\omega$. Since $\omega$ is normalized, that is, $\omega(0,h)=\omega(g,0)=1$, the condition is automatic unless $g=h=k=1$, where it reads $\lambda=\lambda$. So $J'$ is given by the normalized $2$-cocycle $\omega$ on $\Z/2$.
\item For $\tau\colon\Z/2\to\CC^\times$, let $\gamma\colon F'\Rightarrow\Id_\cF$ be the natural isomorphism with $\gamma_U=\tau(p(U))\,\id_U$ for homogeneous $U$. For homogeneous $U,V$ of degrees $g,h$, the tensor condition $\gamma_{U\otimes V}J'_{U,V}=\gamma_U\otimes\gamma_V$ in~\eqref{eq:monoidal-nat} reads $\tau(g+h)\,\omega(g,h)\,\id_{U\otimes V}=\tau(g)\,\tau(h)\,\id_{U\otimes V}$. Hence it holds for all $U,V$ if and only if $\omega=\delta\tau$, where $(\delta\tau)(g,h)=\tau(g)\,\tau(h)\,\tau(g+h)^{-1}$. Since $\iota^{F'}=\id_\one$, the unit condition reads $\tau(0)=1$, and it follows from $\omega=\delta\tau$ because $\omega(0,0)=1$ and $(\delta\tau)(0,0)=\tau(0)$. So $\gamma$ is monoidal if and only if $\omega=\delta\tau$. For $\tau(0)=1$ and $\tau(1)=t$, we have $(\delta\tau)(g,h)=t^{2gh}$, so $\omega=\delta\tau$ if and only if $t^2=\lambda$. Every normalized $2$-cocycle on $\Z/2$ is of the form $(g,h)\mapsto\lambda^{gh}$ with $\lambda\in\CC^\times$, and $\lambda$ has a square root in $\CC^\times$, so we recover $H^2(\Z/2;\CC^\times)=0$.
\end{enumerate}
\end{remark}

\section{Fermion attachment}\label{sec:attachment}

We keep the notation of Section~\ref{sec:vortex}, where we represented every element of $\ker r$ by a normalized pair $(\Phi,\id_\psi)$ of module type $R_\one$ or $R_\psi$ (Proposition~\ref{prop:module-type}) and showed that $(\Phi,\id_\psi)\cong(\Id_\cE,\id_\psi)$ when the module type is $R_\one$ (Lemma~\ref{lem:rigidity}). In this section, we construct a normalized pair of module type $R_\psi$, namely $(\alpha,\id_\psi)$, where $\alpha$ is fermion attachment. It is the identity on $\cC$ and attaches $\psi$ to every vortex. We show that $(\alpha,\id_\psi)\not\cong(\Id_\cE,\id_\psi)$ and $(\alpha^2,\id_\psi)\cong(\Id_\cE,\id_\psi)$. Composing a normalized pair of module type $R_\psi$ with $(\alpha,\id_\psi)$, we obtain module type $R_\psi R_\psi\cong R_\one\cong\Id_\cV$. Thus, by Lemma~\ref{lem:rigidity}, every normalized pair of module type $R_\psi$ is isomorphic to $(\alpha,\id_\psi)$ (Corollary~\ref{cor:alpha-rigid}).

For super-modular categories, fermion attachment is already known to give an element of order two in $\ker r$~\cite{Aasen:2021vva}, \cite[\S VI.B.2]{Bulmash:2021ryq}. This is the lower bound in the Main Theorem. We give a tensor-categorical form of the construction in~\cite[\S III.A, Eqs.~(111)--(114)]{Aasen:2021vva} (see also~\cite[Eq.~(105)]{Bulmash:2021ryq}), with associators and unitors and in the generality of this paper, including nonunitary categories. We fix an isomorphism $e\colon\psi\otimes\psi\to\one$, and let $\mu_{p,q}$ be as in Lemma~\ref{lem:mu}. For a homogeneous object $X$, we write $p(X)\in\{0,1\}$ for its parity, and add exponents of $\psi$ modulo $2$.

\begin{definition}\label{def:alpha}
On homogeneous objects and morphisms, we define
\begin{equation*}
	\alpha(X):=\begin{cases}X,&p(X)=0,\\
	X\otimes\psi,&p(X)=1,
	\end{cases}\qquad\alpha(f):=\begin{cases}
	f,&p(X)=p(X')=0,\\
	f\otimes\id_\psi,&p(X)=p(X')=1
	\end{cases}
\end{equation*}
for $f\colon X\to X'$, and extend by additivity. We first define a tensor structure on the functor $\widetilde\alpha\colon X\mapsto X\otimes\psi^{p(X)}$, where it is given by a single formula for all parities. Since $\cE$ is not assumed to be strict, $\widetilde\alpha(b)=b\otimes\one$ is only isomorphic to $b$ for $b\in\cC$, so we then transport this structure to $\alpha$, which is the identity on $\cC$. On morphisms, $\widetilde\alpha(f)=f\otimes\id_{\psi^{p(X)}}$, and we put $\widetilde\iota=r_\one^{-1}\colon\one\to\one\otimes\one$. For $p=p(X)$, $q=p(Y)$, we abbreviate $P=\psi^p$ and $Q=\psi^q$. We define
\begin{equation*}
	K_{X,Y}\colon (X\otimes P)\otimes(Y\otimes Q) \longrightarrow (X\otimes Y)\otimes\psi^{p+q}
\end{equation*}
by the following composite, read from right to left:
\begin{equation}\label{eq:K}
\begin{aligned}
	K_{X,Y}:={}&(\id_{X\otimes Y}\otimes\mu_{p,q})\circ a^{-1}_{X,Y,P\otimes Q}\circ(\id_X\otimes a_{Y,P,Q})\\
	&\circ\bigl(\id_X\otimes(c_{P,Y}\otimes\id_Q)\bigr)\circ(\id_X\otimes a^{-1}_{P,Y,Q})\circ a_{X,P,Y\otimes Q}.
\end{aligned}
\end{equation}
Thus $K$ moves the attached $P$ past $Y$ and then multiplies $P$ and $Q$. The natural isomorphism
\begin{equation*}
	\vartheta_X\colon\widetilde\alpha(X)\longrightarrow\alpha(X),\qquad \vartheta_X=\begin{cases}
	r_X,&p(X)=0,\\
	\id_{X\otimes\psi},&p(X)=1,
	\end{cases}
\end{equation*}
transports this structure to
\begin{equation}\label{eq:J-alpha}
	J^\alpha_{X,Y}:=\vartheta_{X\otimes Y}\circ K_{X,Y}\circ(\vartheta_X^{-1}\otimes\vartheta_Y^{-1}),\qquad \iota^\alpha:=\vartheta_\one\circ\widetilde\iota=\id_\one.
\end{equation}
We call $(\alpha,J^\alpha,\iota^\alpha)$ \emph{fermion attachment}. In particular, since $p(\psi)=0$, we have $\alpha(\psi)=\psi$.
\end{definition}

Because $\vartheta_X=r_X$ for $p(X)=0$, $\alpha$ is literally the identity on the neutral component. In particular, for $b,b'\in\cC$ and $X\in\cV$, we obtain from coherence (Convention~\ref{conv:coherence}) and the identity $c_{\one,Y}=r_Y^{-1}l_Y$ for $Y\in\cE$
\begin{equation}\label{eq:attachment-mixed}
	\begin{aligned}
	J^\alpha_{b,b'}&=\id_{b\otimes b'},\\
	J^\alpha_{b,X}&=a^{-1}_{b,X,\psi},\\
	J^\alpha_{X,b}&=a^{-1}_{X,b,\psi}\circ(\id_X\otimes c_{\psi,b})\circ a_{X,\psi,b}.
	\end{aligned}
\end{equation}

We now show that $\alpha$ is a braided tensor autoequivalence of $\cE$, and that $(\alpha,\id_\psi)$ is a normalized pair of module type $R_\psi$, so its class lies in $\ker r$.

\begin{proposition}\label{prop:alpha}
With the tensor structure~\eqref{eq:J-alpha}, $\alpha$ is a braided tensor autoequivalence of $\cE$. Moreover:
\begin{enumerate}
\item $\alpha|_\cC=\Id_\cC$ as a tensor functor;
\item $\alpha(\cV)\subseteq\cV$, and $\alpha|_\cV=R_\psi$ as left $\cC$-module functors, with module structure $J^\alpha_{b,X}=a^{-1}_{b,X,\psi}$;
\item $(\alpha,\id_\psi)$ is a normalized LR autoequivalence, and $[\alpha_\psi]:=[\alpha,\id_\psi]$ lies in $\ker r$.
\end{enumerate}
\end{proposition}

\begin{proof}
We first check the tensor and braided structures on $\widetilde\alpha$, before transport. Each $K_{X,Y}$ is an isomorphism by~\eqref{eq:K}, since it is a composite of associators, a braiding and $\mu_{p,q}$, which is a unitor or $e$, tensored with identities. Naturality of $K$ follows from naturality of $a$ and $c$ and functoriality of $\otimes$. By Convention~\ref{conv:graded}, it suffices to check homogeneous objects and morphisms. From the triangle axiom and the identities $c_{\one,X}=r_X^{-1}l_X$ and $c_{X,\one}=l_X^{-1}r_X$, we obtain the unit conditions~\eqref{eq:functor-units} for $(\widetilde\alpha,K,\widetilde\iota)$
\begin{align*}
	\widetilde\alpha(l_X)\circ K_{\one,X}\circ(\widetilde\iota\otimes\id_{\widetilde\alpha(X)})&=l_{\widetilde\alpha(X)},\\
	\widetilde\alpha(r_X)\circ K_{X,\one}\circ(\id_{\widetilde\alpha(X)}\otimes\widetilde\iota)
	&=r_{\widetilde\alpha(X)}.
\end{align*}

For homogeneous $X,Y,Z$ of parities $p,q,r$, the monoidal axiom~\eqref{eq:monoidal-axiom} for $(\widetilde\alpha,K)$ at $(X,Y,Z)$ reads
\begin{equation*}
	\widetilde\alpha(a_{X,Y,Z})\circ K_{X\otimes Y,Z}\circ(K_{X,Y}\otimes\id_{\widetilde\alpha(Z)})=K_{X,Y\otimes Z}\circ(\id_{\widetilde\alpha(X)}\otimes K_{Y,Z})\circ a_{\widetilde\alpha(X),\widetilde\alpha(Y),\widetilde\alpha(Z)}.
\end{equation*}
In the following diagrams, we omit $\otimes$ and canonical isomorphisms as in Convention~\ref{conv:coherence}. By~\eqref{eq:K}, the two sides of the monoidal axiom are the outer paths of the diagram
\begin{equation*}
\begin{tikzpicture}[baseline=(current bounding box.center),every node/.style={font=\small}]
\node (A) at (0,0) {$X\psi^pY\psi^qZ\psi^r$};
\node (B) at (4.4,0) {$XY\psi^p\psi^qZ\psi^r$};
\node (C) at (8.8,0) {$XY\psi^{p+q}Z\psi^r$};
\node (D) at (0,-1.7) {$X\psi^pYZ\psi^q\psi^r$};
\node (E) at (4.4,-1.7) {$XY\psi^pZ\psi^q\psi^r$};
\node (F) at (4.4,-3.4) {$XYZ\psi^p\psi^q\psi^r$};
\node (G) at (8.8,-3.4) {$XYZ\psi^{p+q}\psi^r$};
\node (H) at (0,-5.1) {$X\psi^pYZ\psi^{q+r}$};
\node (I) at (4.4,-5.1) {$XYZ\psi^p\psi^{q+r}$};
\node (J) at (8.8,-5.1) {$XYZ\psi^{p+q+r}$};
\draw[->] (A) -- node[above] {$c_{\psi^p,Y}$} (B);
\draw[->] (B) -- node[above] {$\mu_{p,q}$} (C);
\draw[->] (C) -- node[right] {$c_{\psi^{p+q},Z}$} (G);
\draw[->] (G) -- node[right] {$\mu_{p+q,r}$} (J);
\draw[->] (A) -- node[left] {$c_{\psi^q,Z}$} (D);
\draw[->] (D) -- node[left] {$\mu_{q,r}$} (H);
\draw[->] (H) -- node[below] {$c_{\psi^p,YZ}$} (I);
\draw[->] (I) -- node[below] {$\mu_{p,q+r}$} (J);
\draw[->] (D) -- node[above] {$c_{\psi^p,Y}$} (E);
\draw[->] (B) -- node[right] {$c_{\psi^q,Z}$} (E);
\draw[->] (E) -- node[right] {$c_{\psi^p,Z}$} (F);
\draw[->] (F) -- node[above] {$\mu_{p,q}$} (G);
\draw[->] (F) -- node[right] {$\mu_{q,r}$} (I);
\node[font=\scriptsize] at (2.2,-0.8) {functoriality of $\otimes$};
\node[font=\scriptsize,align=center] at (6.85,-1.7) {hexagon,\\ naturality of $c$};
\node[font=\scriptsize,align=center] at (2.0,-3.4) {hexagon,\\ functoriality\\ of $\otimes$};
\node[font=\scriptsize] at (6.6,-4.25) {\eqref{eq:mu-assoc}};
\end{tikzpicture}
\end{equation*}
whose cells commute for the reasons indicated in the middle of the corresponding cells. Hence the monoidal axiom holds.

For braidedness, let $X,Y$ be homogeneous of parities $p,q$. By~\eqref{eq:K}, $\widetilde\alpha(c_{X,Y})\circ K_{X,Y}$ and $K_{Y,X}\circ c_{\widetilde\alpha(X),\widetilde\alpha(Y)}$ are the outer paths of the diagram
\begin{equation*}
\begin{tikzpicture}[baseline=(current bounding box.center),every node/.style={font=\small}]
\node (A) at (0,0) {$X\psi^pY\psi^q$};
\node (B1) at (4.2,0) {$Y\psi^qX\psi^p$};
\node (B2) at (8.4,0) {$YX\psi^q\psi^p$};
\node (C) at (0,-1.8) {$XY\psi^p\psi^q$};
\node (M) at (4.2,-1.8) {$YX\psi^p\psi^q$};
\node (N) at (0,-3.6) {$XY\psi^{p+q}$};
\node (T) at (8.4,-3.6) {$YX\psi^{p+q}$};
\draw[->] (A) -- node[above] {$c_{X\psi^p,Y\psi^q}$} (B1);
\draw[->] (B1) -- node[above] {$c_{\psi^q,X}$} (B2);
\draw[->] (B2) -- node[right] {$\mu_{q,p}$} (T);
\draw[->] (A) -- node[left] {$c_{\psi^p,Y}$} (C);
\draw[->] (C) -- node[left] {$\mu_{p,q}$} (N);
\draw[->] (N) -- node[below] {$c_{X,Y}$} (T);
\draw[->] (C) -- node[above] {$c_{X,Y}$} (M);
\draw[->] (M) -- node[above left] {$c_{\psi^p,\psi^q}$} (B2);
\draw[->] (M) -- node[below left] {$\mu_{p,q}$} (T);
\node[font=\scriptsize,align=center] at (2.1,-0.9) {$(-1)^{pq}$: hexagon,\\ \eqref{eq:parity}};
\node[font=\scriptsize,align=center] at (7.35,-1.8) {$(-1)^{pq}$:\\ \eqref{eq:mu-comm}};
\node[font=\scriptsize] at (3.3,-2.7) {functoriality of $\otimes$};
\end{tikzpicture}
\end{equation*}
The two cells marked $(-1)^{pq}$ commute up to this sign, and the remaining cell commutes, for the reasons indicated in the middle of the corresponding cells. The two signs cancel, so~\eqref{eq:braided-functor} holds for $(\widetilde\alpha,K)$, namely
\begin{equation*}
	\widetilde\alpha(c_{X,Y})\circ K_{X,Y}=K_{Y,X}\circ c_{\widetilde\alpha(X),\widetilde\alpha(Y)}.
\end{equation*}
By transport along $\vartheta$, $\alpha$ becomes a braided tensor functor. Moreover, $\alpha$ is an autoequivalence of $\cE$, since it is the identity on $\cC$ and $R_\psi$ on $\cV$, and $R_\psi$ is an autoequivalence of $\cV$ with quasi-inverse $R_\psi$, because $\psi\otimes\psi\cong\one$.

It remains to prove (1)--(3).

(1) On objects and morphisms of $\cC$, $\alpha$ is the identity. By~\eqref{eq:attachment-mixed}, $J^\alpha_{b,b'}=\id_{b\otimes b'}$ for $b,b'\in\cC$, and $\iota^\alpha=\id_\one$ by~\eqref{eq:J-alpha}.

(2) On objects and morphisms of $\cV$, $\alpha$ is the functor $R_\psi=-\otimes\psi\colon\cV\to\cV$ of Section~\ref{subsec:morita}, so $\alpha(\cV)\subseteq\cV$. By~\eqref{eq:attachment-mixed}, $J^\alpha_{b,X}=a^{-1}_{b,X,\psi}$ for $b\in\cC$ and $X\in\cV$, which is the module structure $s^{R_\psi}_{b,X}$ of $R_\psi$ (Section~\ref{subsec:morita}). Hence $\alpha|_\cV=R_\psi$ as left $\cC$-module functors.

(3) By (1), $(\alpha,\id_\psi)$ is a normalized LR pair of $\cE$ restricting to the LR pair $(\Id_\cC,\id_\psi)$ of $\cC$. Hence $[\alpha_\psi]\in\ker r$ by Definition~\ref{def:restriction}.
\end{proof}

Next we show that there is no morphism of pairs $(\alpha,\id_\psi)\to(\Id_\cE,\id_\psi)$. Such a morphism would be a monoidal natural isomorphism $\alpha\Rightarrow\Id_\cE$ whose component at $\psi$ is $\id_\psi$. We show that every monoidal natural isomorphism $\alpha\Rightarrow\Id_\cE$ has component $-\id_\psi$ at $\psi$ instead.

\begin{proposition}\label{prop:alpha-nontrivial}
Every monoidal natural isomorphism $\eta\colon\alpha\Rightarrow\Id_\cE$ satisfies $\eta_\psi=-\id_\psi$. Consequently $[\alpha_\psi]\neq1$ in $\AutLR(\cE)$.
\end{proposition}

\begin{proof}
Let $\eta\colon\alpha\Rightarrow\Id_\cE$ be a monoidal natural isomorphism, and choose a nonzero simple $X\in\cV$. Since $p(X)=1$ and $p(\psi)=0$, we obtain from~\eqref{eq:attachment-mixed}
\begin{equation}\label{eq:J-X-psi}
	J^\alpha_{X,\psi}=a^{-1}_{X,\psi,\psi}\circ(\id_X\otimes c_{\psi,\psi})\circ a_{X,\psi,\psi}=-\id_{(X\otimes\psi)\otimes\psi}.
\end{equation}
Consider the diagram
\begin{equation*}
\begin{tikzpicture}[baseline=(current bounding box.center),every node/.style={font=\small}]
\node (A) at (0,0) {$\alpha(X)\otimes\alpha(\psi)$};
\node (B) at (4.2,0) {$\alpha(X\otimes\psi)$};
\node (C) at (8.4,0) {$\alpha(X)$};
\node (E) at (4.2,-1.8) {$X\otimes\psi$};
\node (F) at (8.4,-1.8) {$X$};
\draw[->] (A) -- node[above] {$J^\alpha_{X,\psi}=-\id$} (B);
\draw[->] (B) -- node[above] {$\alpha(\eta_X)$} (C);
\draw[->] (A) -- node[below left] {$\eta_X\otimes\eta_\psi$} (E);
\draw[->] (B) -- node[right] {$\eta_{X\otimes\psi}$} (E);
\draw[->] (C) -- node[right] {$\eta_X$} (F);

\draw[->] (E) -- node[below] {$\eta_X$} (F);
\node[font=\scriptsize] at (2.95,-0.55) {\eqref{eq:monoidal-nat}};
\node[font=\scriptsize,align=center] at (6.3,-0.9) {naturality\\ of $\eta$};
\end{tikzpicture}
\end{equation*}
Here $\alpha(X)\otimes\alpha(\psi)=\alpha(X\otimes\psi)=(X\otimes\psi)\otimes\psi$ and $\alpha(X)=X\otimes\psi$, since $X$ and $X\otimes\psi$ are odd and $\psi$ is even. The triangle commutes by~\eqref{eq:monoidal-nat} for $\eta$ at $(X,\psi)$, since $\Id_\cE$ has trivial tensor structure, and the square commutes by naturality of $\eta$ at $\eta_X\colon X\otimes\psi\to X$. Moreover, $\alpha(\eta_X)=\eta_X\otimes\id_\psi$ by Definition~\ref{def:alpha}, since $\eta_X$ is a morphism of $\cV$. By the square, since $\eta_X$ is invertible, $\eta_{X\otimes\psi}=\eta_X\otimes\id_\psi$. By the triangle and~\eqref{eq:J-X-psi}, we get $\eta_X\otimes\eta_\psi=-\eta_X\otimes\id_\psi$. Precomposing with $\eta_X^{-1}\otimes\id_\psi$ and using faithfulness of $X\otimes-$ (Lemma~\ref{lem:faithful}(1)), we get $\eta_\psi=-\id_\psi$. So there is no morphism of LR pairs $(\alpha,\id_\psi)\to(\Id_\cE,\id_\psi)$, since its component at $\psi$ would have to be $\id_\psi$ by Definition~\ref{def:LR}.
\end{proof}

We next apply relative rigidity (Lemma~\ref{lem:rigidity}) to $\alpha^2$.

\begin{corollary}\label{cor:alpha-square}
There is a monoidal natural isomorphism $\zeta\colon\alpha\circ\alpha\Rightarrow\Id_\cE$ with $\zeta_b=\id_b$ for all $b\in\cC$. Hence $[\alpha_\psi]^2=1$.
\end{corollary}

\begin{proof}
By Proposition~\ref{prop:alpha} and~\eqref{eq:composite}, $\alpha^2|_\cC=\Id_\cC$ as a tensor functor and $\alpha^2|_\cV=R_\psi R_\psi$, with module structure
\begin{equation*}
	J^{\alpha^2}_{b,X}=\alpha(J^\alpha_{b,X})\circ J^\alpha_{b,X\otimes\psi}=(a^{-1}_{b,X,\psi}\otimes\id_\psi)\circ a^{-1}_{b,X\otimes\psi,\psi}\qquad(b\in\cC,\ X\in\cV).
\end{equation*}
The morphisms
\begin{equation*}
	\beta_X:=r_X\circ(\id_X\otimes e)\circ a_{X,\psi,\psi}\colon (X\otimes\psi)\otimes\psi\longrightarrow X\qquad(X\in\cV)
\end{equation*}
form a natural isomorphism $R_\psi R_\psi\Rightarrow\Id_\cV$, since $a$ and $r$ are natural and $e$ is invertible. Since $\Id_\cV$ has the identity module structure, condition~\eqref{eq:module-nat} for $\beta$ implies that the two outer paths of the diagram
\begin{equation*}
\begin{tikzpicture}[baseline=(current bounding box.center),every node/.style={font=\small}]
\node (A) at (0,0) {$b\otimes((X\otimes\psi)\otimes\psi)$};
\node (E) at (4.4,0) {$(b\otimes(X\otimes\psi))\otimes\psi$};
\node (F) at (9.4,0) {$((b\otimes X)\otimes\psi)\otimes\psi$};
\node (B) at (0,-1.7) {$b\otimes(X\otimes(\psi\otimes\psi))$};
\node (G) at (9.4,-1.7) {$(b\otimes X)\otimes(\psi\otimes\psi)$};
\node (C) at (0,-3.4) {$b\otimes(X\otimes\one)$};
\node (H) at (9.4,-3.4) {$(b\otimes X)\otimes\one$};
\node (D) at (4.7,-4.9) {$b\otimes X$};
\draw[->] (A) -- node[above] {$a^{-1}_{b,X\otimes\psi,\psi}$} (E);
\draw[->] (E) -- node[above] {$a^{-1}_{b,X,\psi}\otimes\id_\psi$} (F);
\draw[->] (A) -- node[right] {$\id_b\otimes a_{X,\psi,\psi}$} (B);
\draw[->] (F) -- node[left] {$a_{b\otimes X,\psi,\psi}$} (G);
\draw[->] (G) -- node[above] {$a_{b,X,\psi\otimes\psi}$} (B);
\draw[->] (B) -- node[right] {$\id_b\otimes(\id_X\otimes e)$} (C);
\draw[->] (G) -- node[left] {$\id_{b\otimes X}\otimes e$} (H);
\draw[->] (H) -- node[above] {$a_{b,X,\one}$} (C);
\draw[->] (C) -- node[below left] {$\id_b\otimes r_X$} (D);
\draw[->] (H) -- node[below right] {$r_{b\otimes X}$} (D);
\node[font=\scriptsize] at (4.7,-0.85) {pentagon};
\node[font=\scriptsize,align=center] at (4.7,-2.55) {naturality of $a$};
\node[font=\scriptsize] at (4.7,-4.0) {coherence};
\end{tikzpicture}
\end{equation*}
agree. Its cells commute for the reasons indicated in the middle of the corresponding cells. Here the pentagon axiom is applied to $(b,X,\psi,\psi)$, and the bottom cell commutes by Convention~\ref{conv:coherence}, since both of its paths are canonical isomorphisms. Hence $\beta$ is a module natural isomorphism. Thus, by Lemma~\ref{lem:rigidity}, there is a monoidal isomorphism $\alpha^2\Rightarrow\Id_\cE$ equal to the identity on $\cC$. In particular its component at $\psi$ is $\id_\psi$, so it is a morphism of LR pairs, and $[\alpha_\psi]^2=1$.
\end{proof}

Finally we deduce the analogue of Lemma~\ref{lem:rigidity} for module type $R_\psi$.

\begin{corollary}\label{cor:alpha-rigid}
Let $(\Phi,\id_\psi)$ be a normalized pair such that
$\Phi|_\cV\cong R_\psi$ as left $\cC$-module functors. Then there
is a monoidal natural isomorphism $\nu\colon\Phi\Rightarrow\alpha$
with $\nu_b=\id_b$ for all $b\in\cC$. In particular, $(\Phi,\id_\psi)\cong(\alpha,\id_\psi)$.
\end{corollary}

\begin{proof}
The composite pair $(\alpha\Phi,\id_\psi)$ is normalized. We choose a module natural isomorphism $t\colon\Phi|_\cV\Rightarrow R_\psi$. Then $\alpha(t)$, with components $\alpha(t_X)=t_X\otimes\id_\psi$, is a module natural isomorphism $(\alpha\Phi)|_\cV\Rightarrow R_\psi R_\psi$ for the composite module structures by~\eqref{eq:module-nat} for $t$ and naturality of $J^\alpha$. Composing with the module natural isomorphism $\beta\colon R_\psi R_\psi\Rightarrow\Id_\cV$ from the proof of Corollary~\ref{cor:alpha-square}, we obtain a module natural isomorphism $\beta\circ\alpha(t)\colon(\alpha\Phi)|_\cV\Rightarrow\Id_\cV$. Consequently, by Lemma~\ref{lem:rigidity}, there is a monoidal natural isomorphism $\xi\colon\alpha\Phi\Rightarrow\Id_\cE$ with $\xi_b=\id_b$ on $\cC$. We choose $\zeta$ as in Corollary~\ref{cor:alpha-square}, and put
\begin{equation*}
	\nu_X:=\alpha(\xi_X)\circ\zeta_{\Phi(X)}^{-1}\colon\Phi(X)\longrightarrow\alpha\alpha\Phi(X)\longrightarrow\alpha(X).
\end{equation*}
This is the composite of the monoidal natural isomorphisms $(\zeta\ast\id_\Phi)^{-1}$ and $\id_\alpha\ast\xi$, hence is monoidal. For $b\in\cC$, we have $\nu_b=\alpha(\xi_b)\zeta_b^{-1}=\id_b$ by normalization. In particular $\nu_\psi=\id_\psi$, as required for an LR morphism.
\end{proof}

In particular, the odd--odd tensorators of a normalized pair $(\Phi,\id_\psi)$ of module type $R_\psi$ carry no freedom beyond those of $\alpha$. Indeed, let $\nu$ be as in Corollary~\ref{cor:alpha-rigid}. For $X,Y\in\cV$ we have $X\otimes Y\in\cC$, so $\nu_{X\otimes Y}=\id$, and~\eqref{eq:monoidal-nat} for $\nu$ at $(X,Y)$ reads
\begin{equation*}
	J^\Phi_{X,Y}=J^\alpha_{X,Y}\circ(\nu_X\otimes\nu_Y)\qquad(X,Y\in\cV).
\end{equation*}

Moreover, $[\alpha_\psi]$ does not depend on the choice of $e$. Indeed, for another isomorphism $e'\colon\psi\otimes\psi\to\one$, the pair constructed from $e'$ is normalized and its restriction to $\cV$ is $R_\psi$ as a left $\cC$-module functor, by Proposition~\ref{prop:alpha}. Hence it is isomorphic to $(\alpha,\id_\psi)$ by Corollary~\ref{cor:alpha-rigid}.


\section{Proof of the Main Theorem}\label{sec:proof}

We first prove the surjectivity of $r$.
\begin{proposition}\label{prop:surjectivity}
Let $\cC$ be a slightly degenerate braided fusion category over $\CC$, and let $\cC\subset\cE$ be a fixed minimal nondegenerate extension. Every $\CC$-linear braided tensor autoequivalence $F$ of $\cC$ admits a braided tensor autoequivalence $H$ of $\cE$ and a monoidal natural isomorphism $H|_{\cC}\Rightarrow F$. Consequently,
\begin{equation*}
	r\colon\AutLR(\cE)\longrightarrow\AutLR(\cC)
\end{equation*}
is surjective. The same holds for $\AutLRu$ when the extension is unitary.
\end{proposition}

\begin{proof}
Let $d=\FPdim(\cC)$ and $\cD=\cE^{\mathrm{rev}}\boxtimes\cE$. The graph algebra
\begin{equation}\label{eq:surj-graph}
 A=A_F=\bigoplus_{X\in\Irr(\cC)}X^*\boxtimes F(X)
\end{equation}
is connected \'etale of dimension $d$~\cite[\S3.2]{Davydov:2011dqx}. By standard condensation formulas, the category $\mathcal K=\cD_A^0$ of local $A$-modules~\cite[\S3.5]{Davydov:2010kfz} is nondegenerate, with
\begin{equation}\label{eq:surj-residual}
	\FPdim(\mathcal K)=4,\qquad [\mathcal K]=[\cD]=0
\end{equation}
in the Witt group~\cite[Cors.~3.30, 3.32 and Prop.~5.4]{Davydov:2010kfz}.

The objects $g_1=\psi\boxtimes\one$ and $g_2=\one\boxtimes\psi$ centralize $A$, because $\psi$ centralizes $\cC$. Their free modules $h_i=g_i\otimes A$ are therefore local~\cite[Lem.~3.15]{Davydov:2010kfz}, and since the free-module functor from the centralizer of $A$ to $\cD_A^0$ is braided~\cite[\S\S3.4--3.5]{Davydov:2010kfz}, each $h_i$ satisfies $h_i\otimes_A h_i\cong\one$ and $c_{h_i,h_i}=c_{g_i,g_i}=-\id$. Moreover, $h_1$ and $h_2$ are nontrivial because $g_i$ would be a summand of $A$ if $g_i\otimes A\cong A$, whereas a summand $X^*\boxtimes F(X)$ with a trivial factor has $X\cong\one$ and is therefore $\one\boxtimes\one$. Hence $h_1$ and $h_2$ are fermions in $\mathcal K$. So $\langle h_1\rangle\simeq\sVec$, and $\mathcal K$, being nondegenerate of dimension $4=2\FPdim(\sVec)$, is a minimal nondegenerate extension of $\sVec$. By injectivity of the Witt-class map for such extensions~\cite[Thm.~5.17 and the nonunitary discussion in \S5.3]{Lan:2016rcq}, we identify $\mathcal K$ with the toric code using~\eqref{eq:surj-residual}. In particular, $h_1\cong h_2$ is its unique fermion.

Since the toric code is pointed, we can choose a nontrivial invertible boson $b\in\mathcal K$, namely $e$ or $m$~\cite[\S5.2, Eq.~(41)]{Kitaev:2005hzj}. The subcategory $\{\one,b\}$ is Tannakian, and it equals its own centralizer because that centralizer has dimension $4/2=2$~\cite[\S2.2, Eq.~(5)]{Davydov:2011dqx}, so $\one_{\mathcal K}\oplus b$ is a Lagrangian algebra in $\mathcal K$~\cite[Rem.~4.7]{Davydov:2010kfz}. It lifts to a connected \'etale algebra $L=A\oplus U(b)$ in $\cD$~\cite[Prop.~3.16]{Davydov:2010kfz}, where $U$ is the functor that forgets the $A$-module structure. Since $b$ is invertible, $\FPdim(U(b))=\FPdim(A)=d$~\cite[proof of Lem.~3.5]{Davydov:2011dqx}, so $L$ has dimension $2d$ and is Lagrangian in $\mathcal D$.

The parity grading of Lemma~\ref{lem:grading} on each factor gives $\cD$ a $(\Z/2)^2$-grading with $A\in\cC^{\mathrm{rev}}\boxtimes\cC$ in degree $(0,0)$. Since $b$ is a simple $A$-module, $U(b)$ is homogeneous, say of degree $(p,q)$. Its monodromies with $h_1,h_2$, induced from those with $g_1,g_2$, are $(-1)^p,(-1)^q$, both of which equal $-1$ because $h_1\cong h_2$ is the fermion $\varepsilon$ of the toric code, which braids by $-1$ with $e$ and $m$~\cite[\S5.2]{Kitaev:2005hzj}. Hence $(p,q)=(1,1)$. Consequently $L$ has $L_{0,0}=A$ as an algebra and
\begin{equation}\label{eq:surj-factor-intersections}
	L\cap(\cE^{\mathrm{rev}}\boxtimes\one)=\one\boxtimes\one=L\cap(\one\boxtimes\cE).
\end{equation}

By the classification of connected \'etale algebras in a product~\cite[Thm.~3.6 and Prop.~3.7]{Davydov:2011dqx}, $L$ is therefore isomorphic to the graph algebra $A_H$ of a braided autoequivalence $H$ of $\cE$, because the factor algebras are trivial by~\eqref{eq:surj-factor-intersections} and, since $L$ is Lagrangian, both graph subcategories are the whole factors. The isomorphism $L\cong A_H$ preserves the grading, so it restricts to an algebra isomorphism $A_F=L_{0,0}\cong(A_H)_{0,0}$. The summands of $A_H$ in degree $(0,0)$ are the objects $X^*\boxtimes H(X)$ with $X,H(X)\in\cC$. Their total dimension equals $\FPdim(A_F)=d=\sum_{X\in\Irr(\cC)}\FPdim(X)^2$, so $H(\cC)\subseteq\cC$. Hence $H(\cC)=\cC$, because $H$ maps the finitely many simple objects of $\cC$ injectively to simple objects of $\cC$. Since $\cC$ is closed under tensor products, $(A_H)_{0,0}$ is the graph algebra of $H|_{\cC}$, so
\begin{equation*}
	A_{H|_{\cC}}\cong A_F\qquad\text{as algebras}.
\end{equation*}
By~\cite[Thm.~3.6]{Davydov:2011dqx}, applied in $\cC^{\mathrm{rev}}\boxtimes\cC$, graph algebras of non-isomorphic braided autoequivalences are non-isomorphic. Hence there is a monoidal natural isomorphism $\eta\colon H|_{\cC}\Rightarrow F$.

Given an LR pair $(F,\varphi)$, we take $H$ unit-normalized and put $\widetilde\varphi=\varphi\circ\eta_\psi$. By Remark~\ref{rem:LR-meaning}, $(H|_{\cC},\widetilde\varphi)$ is an LR pair and $\eta$ is an isomorphism of pairs $(H|_{\cC},\widetilde\varphi)\to(F,\varphi)$. Condition~\eqref{eq:vertex} for $(H,\widetilde\varphi)$ reads $e\circ(\widetilde\varphi\otimes\widetilde\varphi)=H(e)\circ J^H_{\psi,\psi}$, which is the same equation as for $(H|_{\cC},\widetilde\varphi)$. Hence $(H,\widetilde\varphi)$ is an LR pair of $\cE$, and
\begin{equation*}
	r[H,\widetilde\varphi]=[H|_{\cC},\widetilde\varphi]=[F,\varphi].
\end{equation*}
The unitary statement follows from Lemma~\ref{lem:unitary}.
\end{proof}

\begin{remark}\label{rem:surjectivity-literature}
The existence statement of Proposition~\ref{prop:surjectivity} also follows by combining earlier results. We fix the embedding $\iota\colon\cC\hookrightarrow\cE$. Since $F$ is a braided autoequivalence of $\cC$, the embedding $\iota\circ F$ has the same image as $\iota$, so $(\cE,\iota)$ and $(\cE,\iota\circ F)$ are both marked minimal nondegenerate extensions of $\cC$~\cite[Rem.~4.10]{Lan:2016rcq}. By~\cite[Thm.~5.4]{Lan:2016rcq}, applied with M\"uger center $\sVec$, the group of minimal nondegenerate extensions of $\sVec$ acts freely and transitively on equivalence classes of marked extensions of $\cC$ by the relative product $\star$. Hence there is a minimal nondegenerate extension $\mathcal P$ of $\sVec$ such that
\begin{equation*}
	(\cE,\iota\circ F)\simeq(\cE,\iota)\star\mathcal P
\end{equation*}
as marked extensions. Here $\star$ is the relative product of~\cite[Lem.~4.11 and \S5.1]{Lan:2016rcq}, obtained by condensing the algebra $\one\boxtimes\one\oplus\psi\boxtimes\psi$ in $\cE\boxtimes\mathcal P$. The torsor argument above applies without unitarity~\cite[Rem.~3.1]{Johnson-Freyd:2026galois}.

Write $[\mathcal D]$ for the Witt class of a nondegenerate braided fusion category $\mathcal D$. Changing the marking leaves the underlying category $\cE$ unchanged, while condensation preserves its Witt class \cite[Prop.~5.4]{Davydov:2010kfz}. Hence
\begin{equation*}
	[\cE]=[\cE\star\mathcal P]=[\cE]+[\mathcal P],\qquad [\mathcal P]=0.
\end{equation*}
The Witt-class map on minimal extensions of $\sVec$ is injective~\cite[Thm.~5.17 and \S5.3]{Lan:2016rcq}, so $\mathcal P$ is the toric-code extension $\mathcal Z(\sVec)$, which is the unit of the group of minimal nondegenerate extensions of $\sVec$~\cite[Thm.~4.20]{Lan:2016rcq}. Since the unit acts trivially~\cite[Lem.~4.18]{Lan:2016rcq}, we get $(\cE,\iota\circ F)\simeq(\cE,\iota)$ as marked extensions. Choosing the direction of this equivalence, we obtain $H\colon\cE\to\cE$ with a monoidal isomorphism $H\circ\iota\Rightarrow\iota\circ F$. Transporting the fermion marking as in the proof of Proposition~\ref{prop:surjectivity}, we obtain a preimage under $r$ of every LR class $[F,\varphi]$.
\end{remark}

\begin{proof}[Proof of the Main Theorem]
By Proposition~\ref{prop:alpha}(3), Proposition~\ref{prop:alpha-nontrivial} and Corollary~\ref{cor:alpha-square}, $[\alpha_\psi]$ is an element of order two in $\ker r$. Conversely, let $\kappa\in\ker r$. By Lemma~\ref{lem:normalization}, $\kappa=[\Phi,\id_\psi]$ for a normalized pair $(\Phi,\id_\psi)$, and by Proposition~\ref{prop:module-type}, $\Phi|_\cV\cong R_a$ as module functors for some $a\in\{\one,\psi\}$.

If $a=\one$, then $\Phi$ satisfies the hypotheses of Lemma~\ref{lem:rigidity} with $\cA=\cC$ and $\cM=\cV$, by Definition~\ref{def:normalized} and Lemma~\ref{lem:normalization}. By Lemma~\ref{lem:rigidity}, there is a monoidal isomorphism $\Phi\Rightarrow\Id_\cE$ that is the identity on $\cC$, in particular at $\psi$. So $\kappa=[\Id_\cE,\id_\psi]=1$.

If $a=\psi$, then $(\Phi,\id_\psi)\cong(\alpha,\id_\psi)$ by Corollary~\ref{cor:alpha-rigid}, so $\kappa=[\alpha_\psi]$.

Hence $\ker r=\{[\Id_\cE,\id_\psi],[\alpha_\psi]\}\cong\Z/2$. Surjectivity follows from Proposition~\ref{prop:surjectivity}. Conjugation acts on the normal subgroup $\ker r$ by automorphisms. Since a group of order two has no nontrivial automorphisms, $\ker r$ is central. Hence~\eqref{eq:central-exact} is a central short exact sequence. Each fibre of $r$ is a coset of $\ker r$, so every class in $\AutLR(\cC)$ has exactly two lifts, which differ by $[\alpha_\psi]$. If $\cE$ is unitary, then by Lemma~\ref{lem:unitary} the unitary restriction homomorphism is identified with $r$, and we obtain the same central short exact sequence.
\end{proof}

\begin{remark}\label{rem:obstruction}
The Main Theorem determines the coefficients of the second-layer obstruction of~\cite[\S VI.A]{Bulmash:2021ryq}, which obstructs lifting a symmetry action on $\cC$ to an action on $\cE$. Let a group $G_b$ act on $\cC$ in the sense of~\cite[\S VI.A]{Bulmash:2021ryq}, with classes $[\rho_g]\in\AutLR(\cC)$ satisfying $[\rho_g][\rho_h]=[\rho_{gh}]$. By the Main Theorem, each $[\rho_g]$ has a lift $[\check\rho_g]\in\AutLR(\cE)$, but the lifts need not satisfy the group law. The failure is measured by
\begin{equation*}
	o_2(g,h)=[\check\rho_{gh}][\check\rho_h]^{-1}[\check\rho_g]^{-1}\in\ker r.
\end{equation*}
Writing $g\cdot m=[\check\rho_g]\,m\,[\check\rho_g]^{-1}$ for $g\in G_b$ and $m\in\ker r$, which lies in $\ker r$ since $\ker r$ is normal, we have $o_2(gh,k)\,o_2(g,h)=o_2(g,hk)\,\big(g\cdot o_2(h,k)\big)$ for $g,h,k\in G_b$ by associativity, so $o_2$ is a $2$-cocycle for this action~\cite[Eqs.~(84) and~(89)]{Bulmash:2021ryq}. A priori, its class lies in $H^2(G_b;\ker r)$, which is a group cohomology group only if $\ker r$ is Abelian, since only then is conjugation by the lifts a well-defined action of $G_b$ on $\ker r$. This is the standing assumption of~\cite[\S VI.A]{Bulmash:2021ryq}, and even then the coefficients may be twisted. By the Main Theorem, $\ker r\cong\Z/2$. Hence $\ker r$ is Abelian, and the conjugation action of $G_b$ on $\ker r$ is trivial, because $\Z/2$ has no nontrivial automorphisms. Thus $H^2(G_b;\ker r)=H^2(G_b;\Z/2)$ is group cohomology with trivial coefficients, and $o_2$ defines a class in it. This class vanishes exactly when the lifts can be chosen to form a homomorphism $G_b\to\AutLR(\cE)$. The same identification of coefficient groups applies to the modified obstruction in~\cite[Eqs.~(98)--(100)]{Bulmash:2021ryq}; its vanishing means that the lift classes can be chosen to satisfy the modified multiplication law in~\cite[Eq.~(99)]{Bulmash:2021ryq}.
\end{remark}

\section{Relation with the theory of graded extensions}\label{sec:extension}

In this section, we identify $\ker r$, which is defined through the fermion $\psi$ and the local-vertex condition~\eqref{eq:vertex}, with a group defined purely in terms of the inclusion $\cC\subset\cE$, namely the automorphism group of $\cE$ in the groupoid of minimal nondegenerate extensions of $\cC$ (Proposition~\ref{prop:G1}(2)). We also determine the monoidal automorphisms of $\Id_\cE$ that are trivial on $\cC$ (Proposition~\ref{prop:G1}(3)), and in Remark~\ref{rem:G1-vs-A} we explain the relation with the Main Theorem.

Minimal nondegenerate extensions of $\cC$ form a groupoid $\mathrm{Mext}$. Its objects are the extensions $\cC\subset\cE$. A morphism $\cE\to\cE'$ is a braided equivalence $F\colon\cE\to\cE'$ together with a monoidal natural isomorphism $\eta\colon F|_\cC\Rightarrow\iota'$, where $\iota'\colon\cC\hookrightarrow\cE'$ is the inclusion. Two morphisms, given by $F,\eta$ and by $F',\eta'$, are identified if there is a monoidal isomorphism $\vartheta\colon F\Rightarrow F'$ with $\eta'\circ\vartheta|_\cC=\eta$. Thus $\mathrm{Mext}$ is the fundamental groupoid of the 2-groupoid of minimal nondegenerate extensions~\cite[\S3]{Johnson-Freyd:2026galois}, whose 2-morphisms are the isomorphisms $\vartheta$. Then $\pi_0(\mathrm{Mext})$, the set of isomorphism classes of objects, is the set of equivalence classes of minimal nondegenerate extensions. In the unitary setting, it is the set studied in~\cite[Thm.~5.4]{Lan:2016rcq}. We now describe the automorphism group $\pi_1(\mathrm{Mext},\cE)$ of $\mathrm{Mext}$ at a chosen object $\cE$.

\begin{definition}\label{def:G1}
$\pi_1(\mathrm{Mext},\cE)$ is the group of classes $\mcl{F,\eta}$, where $F$ is a braided autoequivalence of $\cE$ with $F(\cC)\subseteq\cC$ and $\eta\colon F|_\cC\Rightarrow\Id_\cC$ is a monoidal natural isomorphism. Here $\mcl{F,\eta}=\mcl{F',\eta'}$ if there is a monoidal isomorphism $\vartheta\colon F\Rightarrow F'$ with $\eta'\circ\vartheta|_\cC=\eta$, and the product is $\mcl{F,\eta}\,\mcl{G,\eta'}=\mcl{FG,\eta\circ F(\eta')}$. We use double brackets to distinguish these classes from the classes $[F,\varphi]$ in $\AutLR(\cE)$.
\end{definition}

By the horizontal composite in the proof of Lemma~\ref{lem:LR-group}(2), the product is compatible with identifications. Transporting along $\Theta=\eta\oplus\id_{F|_\cV}$ as in the proof of Lemma~\ref{lem:normalization}, we write every class as $\mcl{F,\id}$ with $F|_\cC=\Id_\cC$ as a tensor functor. Two such classes $\mcl{F,\id}$ and $\mcl{F',\id}$ coincide precisely when there is a monoidal isomorphism $F\Rightarrow F'$ that is the identity on $\cC$, and $\mcl{F,\id}\,\mcl{G,\id}=\mcl{FG,\id}$. Hence $\pi_1(\mathrm{Mext},\cE)$ is the fundamental group at $\cE$~\cite[Def.~5.2]{Noohi:2007} of the 2-groupoid $\mathbf{Ex}_{br}(\Z/2,\cC)$ of braided $\Z/2$-extensions of $\cC$, whose 1-cells restrict to $\Id_\cC$ and whose 2-cells restrict to identities~\cite[\S8.1 and Defs.~8.2, 8.15, 8.16]{Davydov-Nikshych:2021}. In particular, it is a group.

\begin{proposition}\label{prop:G1}
Let $\cC\subset\cE$ be a minimal nondegenerate extension.
\begin{enumerate}
\item $\pi_1(\mathrm{Mext},\cE)=\{\mcl{\Id_\cE,\id},\mcl{\alpha,\id}\}\cong\Z/2$.
\item The homomorphism $\pi_1(\mathrm{Mext},\cE)\to\ker r$,
$\mcl{F,\eta}\mapsto[F,\eta_\psi]$, is an isomorphism.
\item The monoidal natural automorphisms of $\Id_\cE$ that restrict to the
identity on $\cC$ are exactly $\id$ and the grading automorphism
$\varepsilon$ with $\varepsilon_X=(-1)^{p(X)}\id_X$. Thus the automorphism group of
$\Id_\cE$ in $\mathbf{Ex}_{br}(\Z/2,\cC)$ is $\Z/2$.
\end{enumerate}
\end{proposition}

\begin{proof}
(1) Let $\mcl{F,\id}\in\pi_1(\mathrm{Mext},\cE)$ with $F|_\cC=\Id_\cC$. By
Proposition~\ref{prop:module-type}, $F|_\cV\cong R_a$ for a unique $a\in\{\one,\psi\}$, and $a$
depends only on the class, because a monoidal isomorphism $F\Rightarrow F'$ that is the
identity on $\cC$ is a module isomorphism $F|_\cV\Rightarrow F'|_\cV$. If $a=\one$, the
isomorphism $F\Rightarrow\Id_\cE$ of Lemma~\ref{lem:rigidity} is the identity on $\cC$, so
$\mcl{F,\id}=1$. If $a=\psi$, the isomorphism $\nu\colon F\Rightarrow\alpha$ of
Corollary~\ref{cor:alpha-rigid} is the identity on $\cC$, so $\mcl{F,\id}=\mcl{\alpha,\id}$. By Corollary~\ref{cor:alpha-square}, $\mcl{\alpha,\id}^2=1$. Finally, $\mcl{\alpha,\id}\neq1$, since a monoidal isomorphism $\alpha\Rightarrow\Id_\cE$ that is the identity on $\cC$ would be a module isomorphism $R_\psi\Rightarrow R_\one$, contradicting Corollary~\ref{cor:morita}(2).

(2) By Remark~\ref{rem:LR-meaning}, $(F,\eta_\psi)$ is an LR autoequivalence, and by
Definition~\ref{def:restriction} its class lies in $\ker r$. An identification $\vartheta$
of $\mcl{F,\eta}$ with $\mcl{F',\eta'}$ satisfies $\eta'_\psi\vartheta_\psi=\eta_\psi$, so it is a
morphism of pairs $(F,\eta_\psi)\to(F',\eta'_\psi)$. Hence the map is well defined. It
is a homomorphism by~\eqref{eq:compose-pairs}, and it is surjective by
Definition~\ref{def:restriction}. Both groups have order two, by (1) and
the Main Theorem, so it is an isomorphism.

(3) Let $\xi$ be a monoidal automorphism of $\Id_\cE$ with $\xi|_\cC=\id$.
By~\eqref{eq:monoidal-nat} at $(b,X)$ with $b\in\cC$ and $X\in\cV$, $\xi|_\cV$ is a module
automorphism of $\Id_\cV\cong R_\one$, hence a scalar $s$ by Corollary~\ref{cor:morita}(2).
By~\eqref{eq:monoidal-nat} at $(X,Y)$ with $X,Y\in\cV$, $\id=s^2$, so $s=\pm1$.
Conversely, $\varepsilon$ is monoidal because parity is additive (Lemma~\ref{lem:grading}).
\end{proof}

\begin{remark}\label{rem:G1-vs-A}
Proposition~\ref{prop:G1} distinguishes two levels of symmetry data. In the unitary setting, we can view $\cC$ as the original fermionic anyon sector and $\cV$ as the vortex sector introduced by gauging fermion parity.

At the first level, $\pi_1(\mathrm{Mext},\cE)$ classifies braided autoequivalences of $\cE$ whose action on $\cC$ is coherently identified with the identity. There are exactly two classes, represented by the identity and by fermion attachment $\alpha$, which attaches $\psi$ to every vortex. By Proposition~\ref{prop:G1}(2), this group is isomorphic to $\ker r$. A class $\mcl{F,\eta}$ includes a monoidal natural isomorphism $\eta\colon F|_\cC\Rightarrow\Id_\cC$, which records how the transformed anyons in $\cC$ are identified with the original ones, compatibly with fusion. A class $[F,\varphi]\in\ker r$ retains only the fermion identification $\varphi=\eta_\psi$, and requires only that a compatible $\eta$ exist. So we obtain a surjection $\mcl{F,\eta}\mapsto[F,\eta_\psi]$ by forgetting the other components of $\eta$. It is injective because both groups have exactly two elements, by Proposition~\ref{prop:G1}(1) and the Main Theorem.

At the second level, by Proposition~\ref{prop:G1}(3), the only monoidal natural automorphism of $\Id_\cE$ other than the identity that is trivial on $\cC$ is
\begin{equation*}
	\varepsilon_X=\begin{cases}\id_X, & X\in\cC,\\ -\id_X, & X\in\cV.\end{cases}
\end{equation*}
The common vortex sign is compatible with fusion because two vortices fuse into $\cC$, so their two signs multiply to $+1$. This automorphism leaves every anyon label unchanged and measures an ambiguity in the identification of symmetry actions. In particular, if $\xi\colon F\Rightarrow\Id_\cE$ is a monoidal natural isomorphism with $\xi|_\cC=\id$, then the only other such isomorphism is $\varepsilon\circ\xi$. The automorphisms $\id$ and $\varepsilon$ form $\pi_2$ of the 2-groupoid $\mathbf{Ex}_{br}(\Z/2,\cC)$ at $\cE$.

Thus the two copies of $\Z/2$ live at different levels. The group $\pi_1$ classifies symmetry actions, whereas $\pi_2$ consists of automorphisms of the identity action, and only $\pi_1$ enters the group extension~\eqref{eq:central-exact} of the Main Theorem. The abstract values $\pi_1=\pi_2=\Z/2$ are recorded in~\cite[\S3]{Johnson-Freyd:2026galois}. In Proposition~\ref{prop:G1}, we give explicit representatives and identify $\pi_1$ with the LR kernel.
\end{remark}

\begin{remark}\label{rem:DN}
The braided Picard description packages the even and vortex sectors, their fusion, and its coherence data. By~\cite[Thm.~8.18]{Davydov-Nikshych:2021}, $\mathbf{Ex}_{br}(\Z/2,\cC)$ is 2-equivalent to the 2-groupoid of braided monoidal 2-functors $\Z/2\to\underline{\underline{\mathrm{Pic}}}_{br}(\cC)$. The ambient Picard structure has~\cite[\S5.3, Eqs.~(5.19)--(5.20)]{Davydov-Nikshych:2021}
\[
  \pi_1=\Inv(\Zt(\cC))=\{\one,\psi\},
  \qquad \pi_2=\CC^\times \ .
\]

Proposition~\ref{prop:G1} instead computes the homotopy groups of the \emph{mapping 2-groupoid at $\cE$}, both of which are $\Z/2$.

The explicit proof explains the distinction. Proposition~\ref{prop:module-type} gives the two module types $R_{\one}$ and $R_\psi$, corresponding to identity and fermion attachment. After reduction to the identity module type, Lemma~\ref{lem:rigidity} leaves only a common odd--odd tensorator $\lambda\,\id$. Rescaling the vortex identifications by $t$ changes $\lambda$ to $\lambda/t^2$. Thus $t^2=\lambda$ removes the scalar, while the two choices $\pm t$ differ by $\varepsilon$. The residual $\Z/2$ therefore describes compatible identifications,
rather than additional symmetry classes.

This interpretation is supplementary; the kernel proof uses only~\cite[\S8.1, Eq.~(8.5)]{Davydov-Nikshych:2021}
through Lemma~\ref{lem:normalization}.
\end{remark}

\section{Physical interpretation}\label{sec:applications}
In the unitary setting, $\cC$ describes the anyons of a fermionic topological phase, with the local fermion $\psi$ retained as a distinct object. A fixed minimal modular extension $\cE=\cC\oplus\cV$ describes the anyons of the topological phase after gauging the fermion parity \cite[\S1]{Bruillard:2016yio}. The simple objects in $\cV$ are fermion-parity vortices: transporting $\psi$ once around any such vortex produces the phase $-1$. We describe complex-linear topological symmetry classes by LR braided autoequivalences, which specify the action on anyon types and fusion spaces together with a fermion marking satisfying the local-vertex condition. Two representatives are equivalent when related by a monoidal natural isomorphism compatible with their fermion markings. When both markings are $\id_\psi$, this requires the component of the isomorphism at $\psi$ to be $\id_\psi$. The restriction homomorphism $r$ restricts these data to $\cC$. By the Main Theorem, every complex-linear LR symmetry class of $\cC$ has exactly two LR lifts to the fixed extension $\cE$. The two lifts differ by multiplication by $[\alpha_\psi]$, the class of fermion attachment, which binds the local fermion to every vortex.

The proof has three main physical ingredients.
\begin{itemize}
\item A symmetry that is trivial on $\cC$ acts on all vortices by attaching one and the same anyon, $\one$ or $\psi$ (Proposition~\ref{prop:module-type}).
\item If it attaches $\one$, its remaining action on the fusion spaces of two vortices is a single phase, which is removed by rephasing the vortices (Lemma~\ref{lem:rigidity}).
\item Attaching $\psi$ is not a rephasing, because a fermion fused into a vortex has to be exchanged with the bound fermion, at the cost $c_{\psi,\psi}=-1$ (Proposition~\ref{prop:alpha-nontrivial}).
\end{itemize}
We explain these ingredients in three examples.

\subsection{The Laughlin state}\label{subsec:physics-laughlin}

The Laughlin state at filling $\nu=1/k$, with $k$ odd, is described by
\begin{equation}\label{phy:laughlin-action}
	\mathcal L_\cC=-\frac{k}{4\pi}\,a\,da+\frac{1}{2\pi}\,A\,da,
\end{equation}
where $a$ is a compact $U(1)$ gauge field and $A$ is the electromagnetic field, normalized so that the electron has charge $1$~\cite[\S2.1, Eqs.~(2.5)--(2.10)]{Wen:1995qn}. The Wilson line $e^{im\oint a}$ has charge $m/k$ and topological spin $e^{\pi im^2/k}$. The electron, $m=k$, is a fermion that braids trivially with all of these lines, so the anyons of $\cC$ are $m\in\Z/2k$, with $\psi=k$~\cite[\S VI.C.1]{Bulmash:2021hmb}.

Since the electron has charge $1$, fermion parity is the rotation $(-1)^F=e^{i\pi Q}$ of the electromagnetic $U(1)$. We can therefore gauge it in the Lagrangian, by replacing $A$ with $A+c$ for a dynamical gauge field $c$, which is restricted to a $\Z_2$ gauge field by a BF term with a compact field $b$~\cite[\S III, Eqs.~(5)--(6)]{Khan:2016fth}:
\begin{equation}\label{phy:laughlin-gauging}
	-\frac{k}{4\pi}\,a\,da+\frac{1}{2\pi}\,(A+c)\,da-\frac{2}{2\pi}\,b\,dc\;\longrightarrow\;\mathcal L_\cE=-\frac{4k}{4\pi}\,b\,db+\frac{2}{2\pi}\,A\,db.
\end{equation}
If we integrate out $b$, then $c$ becomes a flat $\Z_2$ gauge field, and if we integrate out $c$, then $a=2b$~\cite[\S III.B]{Khan:2016fth}. The gauged theory is $U(1)_{4k}$. Its Wilson lines $W_n=e^{in\oint b}$, $n\in\Z/4k$, have
\begin{equation}\label{phy:gauged-data}
	Q(n)=\frac{n}{2k},\qquad\theta_n=e^{\pi in^2/(4k)},\qquad M_{n,n'}=e^{2\pi inn'/(4k)},
\end{equation}
and, at vanishing electromagnetic background $A=0$, their correlators are computed by the Gaussian integrals
\begin{equation}\label{phy:correlator}
	\Bigl\langle\prod_iW_{n_i}(L_i)\Bigr\rangle=\exp\Bigl(\frac{\pi i}{4k}\sum_{i,j}n_in_j\,\mathrm{Link}(L_i,L_j)\Bigr),
\end{equation}
with the self-linking numbers fixed by the framing. Since $e^{im\oint a}=W_{2m}$, the anyons of $\cC$ are the even lines, while the odd lines are the vortices. Before gauging, a fermion-parity flux is a flux tube of $A$ along a line $L$ with $dA=\pi\,\delta_L$, where $\delta_L$ is the two-form dual to $L$. By the field equation $k\,da=dA$ of~\eqref{phy:laughlin-action}, its $a$-flux is $\pi/k$. Since the electric charge density is $\frac{1}{2\pi}da$ by~\eqref{phy:laughlin-action}, it carries the charge $\frac{1}{2\pi}\cdot\frac{\pi}{k}=\frac{1}{2k}$, as in Laughlin's argument~\cite[\S III.A]{Khan:2016fth}, and the electron $e^{ik\oint a}$ acquires the phase $e^{i\pi}=-1$ around it. With a Wilson line $e^{iq\oint_La}$ the field equation would read $k\,da=2\pi q\,\delta_L$, so the flux tube has the same $a$-flux, charge and braiding as a Wilson line of $a$-charge $q=\frac12$. Such a line is not gauge invariant in $U(1)_k$, whose $a$-charges are integers. In its gauge-invariant form $e^{\frac i2\int_{\Sigma_L}da}=e^{i\pi Q(\Sigma_L)}$, with $\partial\Sigma_L=L$, the line is attached to a surface on which $(-1)^F$ acts, so the flux tube is only an external defect. After gauging, we insert $W_1=e^{i\oint_Lb}$ into the left side of~\eqref{phy:laughlin-gauging} and obtain
\begin{equation}\label{phy:laughlin-flux}
	dc=\pi\,\delta_L,\qquad da=2\,db=\frac1k\bigl(dA+\pi\,\delta_L\bigr),
\end{equation}
so on top of the response $dA/k$ to the background field, $W_1$ carries the $\pi$ flux of the $\Z_2$ gauge field $c$. After $c$ is integrated out, the vortex contributes $db=\frac{\pi}{2k}\,\delta_L$ in $\mathcal L_\cE$, and the half-charge line $e^{\frac i2\oint a}=e^{i\oint b}=W_1$ is now a genuine line. The electron $\psi=W_{2k}$ braids with $W_n$ by $M_{2k,n}=(-1)^n$, which is~\eqref{eq:parity}. So the electron is transparent in $\cC$, and every vortex detects it.

Fermion attachment (Definition~\ref{def:alpha}) binds the electron worldline to every vortex worldline:
\begin{equation}\label{phy:laughlin-attachment}
 \alpha\colon W_n\longmapsto W_n\,W_{2k}^{\,n}=W_{(1+2k)n}.
\end{equation}
It preserves the correlators~\eqref{phy:correlator}, because $(1+2k)^2\equiv1\pmod{8k}$ for odd $k$. Physically, binding the electron multiplies the topological spin of a vortex by $\theta_{2k}=-1$ and by the Aharonov--Bohm phase $-1$ of the electron around the vortex, and the two signs cancel, as in the proof of Proposition~\ref{prop:alpha}. It shifts the charge of every vortex by one electron. Here and below, we consider symmetries of the underlying braided category. Preserving the electromagnetic charge assignment modulo local charge-$2$ bosons is an additional requirement, which $\alpha$ does not satisfy.

Fermion attachment is the only nontrivial symmetry of $U(1)_{4k}$ that is trivial on $\cC$, in agreement with the Main Theorem. Indeed, a symmetry $\Phi$ that is trivial on $\cC$ fixes the even lines, and since $W_n=W_1W_{n-1}$, it is determined on the vortices by the image $W_{1+2m'}$ of the bare vortex $W_1$. Since the symmetry preserves the braiding $M_{2,1}=e^{\pi i/k}$ with the quasiparticle $W_2$, we get $2m'\in\{0,2k\}$, so the symmetry binds either nothing or one electron to every vortex, which is exactly Proposition~\ref{prop:module-type}. If $\Phi$ binds nothing, it fixes every line, and it can only multiply the junctions of Wilson lines by phases. A network $\Gamma$ of Wilson lines whose $b$-charges $n_e$ are conserved at the junctions defines the operator
\begin{equation}\label{phy:network}
 W[\Gamma]=\exp\Bigl(i\sum_en_e\int_eb\Bigr),
\end{equation}
where $e$ runs over the edges of $\Gamma$. It is gauge invariant because the charge is conserved at each junction. Its correlators are Gaussian integrals, as in~\eqref{phy:correlator}, so they do not change when a junction slides along the network. The symmetry multiplies the junction at which $W_{n_1}$ and $W_{n_2}$ merge into $W_{n_1+n_2}$ by a phase $u(n_1,n_2)$. Here $u(n_1,n_2)=1$ unless $n_1$ and $n_2$ are both odd, because the symmetry is trivial on the even lines and on their fusion with vortices, as required by Lemma~\ref{lem:rigidity}(1) and~(3). Now attach an even line $W_{2m}$, drawn dashed, to a junction of two vortex lines, on either incoming leg or on the outgoing leg:
\begin{equation}\label{phy:networks}
\begin{tikzpicture}[baseline=(current bounding box.center),scale=0.9,every node/.style={font=\scriptsize}]
\foreach \x/\k in {0/a,4.4/b,8.8/c} {
\begin{scope}[xshift=\x cm]
\draw[thick] (-0.8,-1) node[below]{$n_1$} -- (0,0);
\draw[thick] (0.8,-1) node[below]{$n_2$} -- (0,0);
\draw[thick] (0,0) -- (0,1);
\end{scope}
}
\draw[thick,dashed] (-1.4,-0.1) node[above]{$2m$} -- (-0.4,-0.5);
\draw[thick,dashed] (4.4+1.4,-0.1) node[above]{$2m$} -- (4.4+0.4,-0.5);
\draw[thick,dashed] (8.8-1.0,1.0) node[above]{$2m$} -- (8.8,0.5);
\node at (2.2,0) {$=$};
\node at (6.6,0) {$=$};
\end{tikzpicture}
\end{equation}
The three networks differ only by sliding the junction of $W_{2m}$, so they are the same operator, and the symmetry multiplies them by the same phase:
\begin{equation}\label{phy:laughlin-rigidity}
	u(n_1+2m,n_2)=u(n_1,n_2+2m)=u(n_1,n_2)
\end{equation}
which is~\eqref{eq:LBR}. Since every odd label is $1$ plus an even label, $u(n_1,n_2)=u(1,1)=:\lambda$ for all pairs of vortices, and it is removed by rephasing every vortex line by $\sqrt\lambda$. This is the conclusion of Lemma~\ref{lem:rigidity} in the Laughlin state. If $\Phi$ binds one electron to every vortex, its composite $\alpha\Phi$ with $\alpha$ binds two, a local boson, so $\alpha\Phi$ binds nothing and is a rephasing. Since $\alpha^2$ is also a rephasing (Corollary~\ref{cor:alpha-square}), $\Phi$ differs by rephasings from $\alpha^2\Phi=\alpha(\alpha\Phi)$ and hence from $\alpha$. So $\Phi$ and $\alpha$ define the same element $[\alpha_\psi]$ of $\ker r$ (Corollary~\ref{cor:alpha-rigid}).

Finally, $\alpha$ moves the vortices, so it is not a rephasing itself. From its tensor structure we also see why fermion attachment survives in general. When an electron is fused into a vortex $W_n$, with $n$ odd, that already carries the bound electron, the two electrons are exchanged:
\begin{equation}\label{phy:laughlin-exchange}
	\alpha(W_n)\,\alpha(\psi)=W_n\,\underbrace{\psi}_{\text{bound}}\,\underbrace{\psi}_{\text{fused in}}\xrightarrow{\ \id\otimes c_{\psi,\psi}\ }W_n\,\underbrace{\psi}_{\text{fused in}}\,\underbrace{\psi}_{\text{bound}}=\alpha(W_n\psi).
\end{equation}
So~\eqref{eq:J-X-psi} becomes the exchange sign $c_{\psi,\psi}=\theta_{2k}=e^{\pi ik}=-1$ of two electrons. To remove $\alpha$ by a rephasing, we would have to rephase the electron by this sign, which is not allowed (Proposition~\ref{prop:alpha-nontrivial}). Therefore $\alpha$ is nontrivial.

\subsection{The \texorpdfstring{$SO(3)_3$}{SO(3)3} theory}\label{subsec:physics-so3}

The spin Chern--Simons theory $SO(3)_3$ is a non-Abelian example in which fermion parity cannot be gauged in the Lagrangian formalism as in the previous example, since $SO(3)_3$ has no $U(1)$ symmetry that contains $(-1)^F$. The background field of $(-1)^F$ is then the spin structure, which enters only through the global definition of the theory and not as a field in its Lagrangian, so there is no analogue of the substitution $A\mapsto A+c$. We label the Wilson lines of the Chern--Simons theory $SU(2)_6$~\cite{Witten:1989} by twice their isospin, $j=0,1,\dots,6$. Their topological spins are $\theta_j=e^{2\pi i\,j(j+2)/32}$~\cite[\S6, Eq.~(6.2)]{Kirillov:2001ti}, and the line $\psi=W_6$, of isospin $3$ and topological spin $\theta_6=e^{3\pi i}=-1$, is a fermion. Then we can obtain $SO(3)_3$ through gauging the fermion line $\psi$ in $SU(2)_6$~\cite[\S3.1, Eq.~(3.15)]{Delmastro:2020dkz}
\begin{equation}\label{phy:so3-quotient}
	SO(3)_3=SU(2)_6/\Z_2^\psi \ .
\end{equation}
Both theories have the Chern--Simons action $\frac{6}{4\pi}\int\operatorname{tr}\bigl(a\,da +\frac23a^3\bigr)$, with the trace in the two-dimensional representation of $\mathfrak{su}(2)\cong\mathfrak{so}(3)$, and they differ only in their gauge bundles. The anyons of $SO(3)_3$ are the even lines $0,2,4,6$, with $\psi=W_6$.

Gauging $(-1)^F$ therefore means summing over spin structures $s$~\cite[\S3.1, Eq.~(3.18)]{Delmastro:2020dkz}. The quotient~\eqref{phy:so3-quotient} sums over insertions of the fermion line $\psi$ along closed curves $\gamma$ in spacetime, taken modulo $2$ because $W_6W_6=W_0$, that is, over $SO(3)$ gauge bundles $P$ with $w_2(P)$ dual to $\gamma$, and the fermion line needs the spin structure. Schematically,
\begin{equation}\label{phy:so3-partition}
	Z_{SO(3)_3}[s]\;\propto\sum_{\gamma\in H_1(M;\Z_2)}(-1)^{q_s(\gamma)}\,Z_{SU(2)_6}[\psi\text{ on }\gamma],
\end{equation}
where $Z_{SU(2)_6}[\psi\text{ on }\gamma]$ is the partition function of $SU(2)_6$ with the Wilson line $W_6$ inserted along $\gamma$, and $(-1)^{q_s(\gamma)}$ is the sign of the fermion line along $\gamma$ in the spin structure $s$. Deforming $\gamma$ changes each factor by a sign, but not their product~\cite[\S2.1]{Bhardwaj:2016clt}. The same formula appears in~\cite[\S2.1, Eq.~(2.7)]{Bhardwaj:2016clt} as a sum over the classes $\beta\in H^2(M;\Z_2)$ Poincar\'e dual to $\gamma$. The spin structure enters the sign only through a factor $(-1)^{\int_M\eta\cup\beta}$, where the cochain $\eta$ represents $s$~\cite[\S1.2, Eq.~(4)]{Gaiotto:2015zta}. So if we shift $s$ by $c\in H^1(M;\Z_2)$, this sign is multiplied by $(-1)^{c(\gamma)}$, and
\begin{equation}\label{phy:so3-gauging}
	\sum_cZ_{SO(3)_3}[s+c]\;\propto\sum_\gamma(-1)^{q_s(\gamma)}\,Z_{SU(2)_6}[\psi\text{ on }\gamma]\sum_c(-1)^{c(\gamma)}\;\propto\;Z_{SU(2)_6}.
\end{equation}
Since $\sum_c(-1)^{c(\gamma)}=|H^1(M;\Z_2)|\,\delta_{\gamma,0}$, only $\gamma=0$ survives. This is the analogue of integrating out $c$ in~\eqref{phy:laughlin-gauging}. In the sum over spin structures, the $SO(3)$ bundles with $w_2(P)\neq0$ cancel, and what remains is the bosonic parent $SU(2)_6$, as for any spin TQFT~\cite[\S1]{Bhardwaj:2016clt}. Indeed, since $\gamma$ is Poincar\'e dual to $w_2(P)$, the weight is $(-1)^{c(\gamma)}=(-1)^{\int_Mc\cup w_2(P)}$, so $c$ is a $\Z_2$ Lagrange multiplier that sets $w_2(P)=0$. Since $W_jW_6=W_{6-j}$, by the balancing identity~\eqref{eq:balancing} with $U=W_j$ and $V=\psi$ (compare~\cite[\S3, Eq.~(3.3)]{Delmastro:2020dkz}), the phase that $\psi$ acquires around $W_j$ is
\begin{equation}\label{phy:so3-detect}
	c_{\psi,W_j}\,c_{W_j,\psi}=\theta_{6-j}\,\theta_j^{-1}\,\theta_6^{-1}=(-1)^j,
\end{equation}
the action of the center $-1\in SU(2)$. So the odd lines $1,3,5$ braid with $\psi$ by $-1$. They are the fermion-parity fluxes, and in $SO(3)_3$ they are not genuine lines, for the same reason as the half-charge line of the Laughlin state. By~\eqref{phy:so3-detect}, a $\psi$ insertion $\gamma$ in~\eqref{phy:so3-partition} changes sign when it crosses an odd line $W_j(L)$. If we twist the spin structure across a surface $\Sigma_L$ with $\partial\Sigma_L=L$, we obtain the same sign, since
\begin{equation}\label{phy:so3-flux}
	(-1)^{q_{s+\Sigma_L}(\gamma)}=(-1)^{q_s(\gamma)}\,(-1)^{\Sigma_L\cdot\gamma}.
\end{equation}
So in $SO(3)_3$ an odd line exists only at the boundary of such a twist, around which the fermion acquires $-1$. This is the analogue of the flux tube $dA=\pi\,\delta_L$, and these lines form the twisted sector of~\cite[\S3]{Delmastro:2020dkz}. After gauging, the twist surface $\Sigma_L$ is no longer observable. Replacing it by another surface $\Sigma'_L$ with the same boundary shifts the spin structure by the class Poincar\'e dual to $\Sigma_L-\Sigma'_L$ in $H^1(M;\Z_2)$, which is summed over in~\eqref{phy:so3-gauging}. So only $L$ remains, and after gauging the odd lines are genuine lines of $SU(2)_6$, the vortices.

Fermion attachment binds the fermion line to every vortex line,
\begin{equation}\label{phy:so3-attachment}
	\alpha\colon W_j\longmapsto W_j\,W_6=W_{6-j}\qquad(j\text{ odd}).
\end{equation}
As for the Laughlin state, binding $\psi$ multiplies the topological spin of a vortex by $\theta_6=-1$ and by the monodromy $(-1)^j=-1$, so $\theta_{6-j}=\theta_j$ by~\eqref{phy:so3-detect}. Thus $\alpha$ exchanges the vortices $1$ and $5$, both of topological spin $e^{3\pi i/16}$, and it fixes $3$~\cite[\S3.1, Eq.~(3.16)]{Delmastro:2020dkz}.

A symmetry $\Phi$ that is trivial on $SO(3)_3$ fixes the even lines. It fixes $\psi$, so by~\eqref{phy:so3-detect} it maps vortices to vortices. For the Laughlin state, preserving the braiding with $W_2$ excluded every attached quasiparticle other than the vacuum and the electron. Here no attachment is excluded by braiding, since the only even lines that can be attached to a vortex without changing its quantum dimension, which a symmetry preserves~\cite[Prop.~4.5.7]{EGNO:2015}, are the invertible ones, $W_0$ and $W_6$, and these braid trivially with every even line. Instead, $\Phi$ is restricted by fusion. Since it fixes $W_2$, it satisfies $\Phi(W_2W_j)=W_2\,\Phi(W_j)$, and by the truncated Clebsch--Gordan rules~\cite[Example~4.10.6]{EGNO:2015}
\begin{equation}\label{phy:so3-chain}
	W_2W_1=W_1+W_3,\qquad W_2W_3=W_1+W_3+W_5,\qquad W_2W_5=W_3+W_5
\end{equation}
we get, for $j=1$ and $j=3$, \begin{align} \Phi(W_1)+\Phi(W_3)&=W_2\,\Phi(W_1),\label{phy:so3-half}\\ \Phi(W_1)+\Phi(W_3)+\Phi(W_5)&=W_2\,\Phi(W_3).\label{phy:so3-threehalves} \end{align} The images $\Phi(W_j)$ are simple vortices, so the left side of~\eqref{phy:so3-half} has two summands, and hence $\Phi(W_1)\neq W_3$ by~\eqref{phy:so3-chain}. If $\Phi(W_1)=W_1$, then $\Phi(W_3)=W_3$ by~\eqref{phy:so3-half} and $\Phi(W_5)=W_5$ by~\eqref{phy:so3-threehalves}. If $\Phi(W_1)=W_5$, then $\Phi(W_3)=W_3$ and $\Phi(W_5)=W_1$ by~\eqref{phy:so3-half} and~\eqref{phy:so3-threehalves}. So $\Phi$ is determined on the vortices by the image of $W_1$, and it binds either nothing or one fermion to every vortex, which is exactly Proposition~\ref{prop:module-type}. If $\Phi$ binds nothing, it fixes every line, and as in~\eqref{phy:networks} it can only multiply the junctions of Wilson-line networks by phases. Unlike in the Laughlin state, a junction of two vortex lines $W_i$ and $W_j$ also carries a channel, an even line $W_l$ in the fusion $W_iW_j$. For example, $W_3W_3=W_0+W_2+W_4+W_6$ has four channels. Sliding the junction of an attached line $W_2$, as in~\eqref{phy:networks}, is now an $F$-move: \begin{equation}\label{phy:so3-fmove}
\begin{tikzpicture}[baseline=(current bounding box.center),scale=0.9,every node/.style={font=\scriptsize}]
\draw[thick] (-1,-1.2) node[below]{$i$} -- (-0.62,-0.744) -- node[left,pos=0.72]{$i'$} (0,0);
\draw[thick] (1,-1.2) node[below]{$j$} -- (0,0);
\draw[thick] (0,0) -- (0,1.1) node[above]{$l$};
\draw[thick,dashed] (-1.72,-0.36) node[above]{$2$} -- (-0.62,-0.744);
\node[font=\normalsize] at (2.65,0) {$=\;\displaystyle\sum_{l'}\,\bigl[F^{2ij}_{l}\bigr]_{i'l'}$};
\begin{scope}[xshift=5.7cm]
\draw[thick] (-1,-1.2) node[below]{$i$} -- (0,0);
\draw[thick] (1,-1.2) node[below]{$j$} -- (0,0);
\draw[thick] (0,0) -- node[right]{$l'$} (0,0.55) -- (0,1.1) node[above]{$l$};
\draw[thick,dashed] (-1.05,1.05) node[above]{$2$} -- (0,0.55);
\end{scope}
\end{tikzpicture}
\end{equation} Here $i,i',j$ are vortices and $l,l'$ are channels. The vortex junction on the left side of~\eqref{phy:so3-fmove} merges the pair $(i',j)$ in the channel $l$, while those on the right side merge the pair $(i,j)$ in the channels $l'$. So the move relates junctions of different pairs and different channels, with the $6j$-symbols $[F^{2ij}_l]_{i'l'}$ of $SU(2)_6$ as coefficients. The symmetry fixes every line, so it commutes with these moves. By Lemma~\ref{lem:rigidity}, the symmetry therefore acts on the junctions of all pairs of vortices, in all their channels, by a single phase $\lambda$, and it is removed by rephasing every vortex line by $\sqrt\lambda$. If $\Phi$ binds one fermion to every vortex, then $\alpha\Phi$ binds two, and $W_6W_6=W_0$, so $\alpha\Phi$ binds nothing. By~\eqref{phy:so3-fmove} and Lemma~\ref{lem:rigidity}, $\alpha\Phi$ is then a rephasing, and as in the Laughlin state $\Phi$ differs from $\alpha$ by rephasings (Corollaries~\ref{cor:alpha-square} and~\ref{cor:alpha-rigid}). Fermion attachment, on the other hand, moves the vortices $1$ and $5$, and its tensor structure carries the exchange sign $c_{6,6}=\theta_6=-1$ of the bound fermion line with a fused-in one, as in~\eqref{eq:J-X-psi}.

\subsection{The Ising category}\label{subsec:physics-ising}

Free fermions with odd Chern number $\nu$\footnote{Here $\nu$ denotes the spectral Chern number in the Bogolyubov--Nambu description.}, for instance a $p+ip$ superconductor, have no anyons besides the local fermion, so $\cC=\sVec$~\cite[\S II.A]{Read:1999fn},~\cite[\S9]{Kitaev:2005hzj},~\cite[Prop.~3.11]{Bruillard:2016yio},~\cite[\S4.4]{Lan:2016rcq}. For $\nu=1$ the low-energy theory is a massive Majorana fermion $\lambda$,
\begin{equation}\label{phy:ising-action}
	\mathcal L_\cC=\frac i2\,\bar\lambda\gamma^\mu\partial_\mu\lambda-\frac m2\,\bar\lambda\lambda,
\end{equation}
with the sign of $m$ for which $\nu=1$~\cite[\S II.B, Eqs.~(21)--(23)]{Read:1999fn},~\cite[\S6.3, Eqs.~(51) and~(55)]{Kitaev:2005hzj}. Since $\lambda$ is a fermion,~\eqref{phy:ising-action} needs a spin structure $s$. More generally, consider the family of invertible spin TQFTs
\[
SO(N)_1=\operatorname{Spin}(N)_1/\mathbb Z_2^\psi,
\]
with \(N\) odd [25, §3.1, Eqs. (3.13)–(3.14)], [29, App. B, Eqs. (B.1)–(B.2)]. As for \(SO(3)_3\), gauging fermion parity by summing over spin structures, as in (8.13), gives \(\operatorname{Spin}(N)_1\), an Ising theory with vortex spin \(\theta_\sigma=e^{\pi iN/8}\). Its braided anyon data agree with those obtained by gauging the phase of Chern number \(\nu\) when \(N\equiv\nu\pmod{16}\); identifying the full chiral phases additionally requires matching their chiral central charges. A fermion-parity flux is a twist of the spin structure across a surface $\Sigma_L$ with $\partial\Sigma_L=L$, so that $\lambda\mapsto-\lambda$ around $L$. For \(\nu=1\), the field equation of~\eqref{phy:ising-action} in this background has an unpaired Majorana zero mode \(\gamma_L\) bound to \(L\) [19, §8.2]. More generally, a fermion-parity flux in a phase with odd \(\nu\) carries a protected unpaired Majorana zero mode. After gauging, the twist is summed over,
and $L$ becomes the vortex $\sigma$, with
\begin{equation}\label{phy:ising-fusion}
 \psi\otimes\sigma\cong\sigma,\qquad
 \sigma\otimes\sigma\cong\one\oplus\psi
\end{equation}
\cite[\S8.3, Eq.~(64)]{Kitaev:2005hzj}. Because of its zero mode, a vortex can absorb a fermion.

Since $\psi\otimes\sigma\cong\sigma$, the functor $-\otimes\psi$ also fixes $\sigma$, so here hypothesis~(3) of Lemma~\ref{lem:rigidity} cannot be seen on objects. Let $F$ be a braided autoequivalence of the Ising category with properties (1)--(3) of Lemma~\ref{lem:rigidity} for $\cA=\sVec$, and let $J'$ be as in the proof of Lemma~\ref{lem:rigidity}. Fix an isomorphism $(p_\one,p_\psi)\colon\sigma\otimes\sigma\to\one\oplus\psi$. Its component $p_b\colon\sigma\otimes\sigma\to b$ is nonzero, so it spans $\Hom(\sigma\otimes\sigma,b)$. Physically, $p_b$ is the vertex at which the two vortices fuse into $b$. Since $p_b\circ J'_{\sigma,\sigma}$ lies in $\Hom(\sigma\otimes\sigma,b)$,
\begin{equation*}
	p_b\circ J'_{\sigma,\sigma}=u_b\,p_b\qquad(b=\one,\psi)
\end{equation*}
for phases $u_\one$ and $u_\psi$, which do not depend on the choice of $p_b$. A priori $u_\one$ and $u_\psi$ may differ, and then the symmetry detects whether the two vortices fuse into $\one$ or into $\psi$. Fix isomorphisms $e\colon\psi\otimes\psi\to\one$ and $f\colon\psi\otimes\sigma\to\sigma$, and put
\begin{equation*}
v:=e\circ(\id_\psi\otimes p_\psi)\circ a_{\psi,\sigma,\sigma},\qquad w:=p_\one\circ(f\otimes\id_\sigma)
\qquad\text{in }\Hom((\psi\otimes\sigma)\otimes\sigma,\one).
\end{equation*}
Since $p_\one$ and $p_\psi$ are nonzero, $\psi\otimes-$ is an autoequivalence, and $e$, $f\otimes\id_\sigma$ and $a_{\psi,\sigma,\sigma}$ are isomorphisms, $v$ and $w$ are nonzero. By~\eqref{eq:LBR} for $a=\psi$ and $X=Y=\sigma$,
\begin{equation*}
a_{\psi,\sigma,\sigma}\circ J'_{\psi\otimes\sigma,\sigma}=(\id_\psi\otimes J'_{\sigma,\sigma})\circ a_{\psi,\sigma,\sigma},
\end{equation*}
and, by naturality of $J'$,
\begin{equation*}
(f\otimes\id_\sigma)\circ J'_{\psi\otimes\sigma,\sigma}=J'_{\sigma,\sigma}\circ(f\otimes\id_\sigma).
\end{equation*}
Hence
\begin{align*}
v\circ J'_{\psi\otimes\sigma,\sigma}&=e\circ(\id_\psi\otimes p_\psi)\circ(\id_\psi\otimes J'_{\sigma,\sigma})\circ a_{\psi,\sigma,\sigma}\\
&=e\circ\bigl(\id_\psi\otimes(p_\psi\circ J'_{\sigma,\sigma})\bigr)\circ a_{\psi,\sigma,\sigma}=u_\psi\,v,\\
w\circ J'_{\psi\otimes\sigma,\sigma}&=p_\one\circ(f\otimes\id_\sigma)\circ J'_{\psi\otimes\sigma,\sigma}
=p_\one\circ J'_{\sigma,\sigma}\circ(f\otimes\id_\sigma)=u_\one\,w.
\end{align*}
Since $(\psi\otimes\sigma)\otimes\sigma\cong\psi\otimes(\sigma\otimes\sigma)\cong\psi\otimes(\one\oplus\psi)\cong\psi\oplus\one$, by Schur's lemma~\cite[Prop.~1.8.4]{EGNO:2015},
\begin{equation*}
\Hom((\psi\otimes\sigma)\otimes\sigma,\one)\cong\Hom(\psi,\one)\oplus\End(\one)=0\oplus\CC.
\end{equation*}
As $v\neq0$, $v$ spans $\Hom((\psi\otimes\sigma)\otimes\sigma,\one)$, so $w=\kappa\,v$ for some $\kappa\in\CC$, and therefore
\begin{equation*}
u_\one\,w=w\circ J'_{\psi\otimes\sigma,\sigma}=\kappa\,(v\circ J'_{\psi\otimes\sigma,\sigma})=\kappa\,u_\psi\,v=u_\psi\,w.
\end{equation*}
As $w\neq0$, it follows that $u_\one=u_\psi$. Hence
\begin{equation*}
(p_\one,p_\psi)\circ J'_{\sigma,\sigma}=(p_\one\circ J'_{\sigma,\sigma},\,p_\psi\circ J'_{\sigma,\sigma})
=u_\one\,(p_\one,p_\psi),
\end{equation*}
and since $(p_\one,p_\psi)$ is an isomorphism, $J'_{\sigma,\sigma}=u_\one\,\id_{\sigma\otimes\sigma}$, which is~\eqref{eq:J-scalar} for the Ising category. Physically, $v$ and $w$ are proportional, $w=\kappa\,v$, so they describe the same state of a fermion and two vortices with total charge $\one$. In $v$ the two vortices fuse into $\psi$, and in $w$ the vortex $\psi\otimes\sigma$, which has absorbed the fermion, fuses with $\sigma$ into $\one$. Thus a symmetry that acts trivially on $\sVec$ and attaches nothing to $\sigma$, that is, has module type $R_\one$ as in Lemma~\ref{lem:rigidity}(1) and~(3), multiplies $\Hom(\sigma\otimes\sigma,\one)$ and $\Hom(\sigma\otimes\sigma,\psi)$ by the same phase, so it cannot distinguish the fusion channels $\one$ and $\psi$ of two vortices. Its remaining phase is removed by rephasing $\sigma$ by $t$ with $t^2=u_\one$, as in the proof of Lemma~\ref{lem:rigidity}. For module type $R_\psi$, the symmetry acts on the zero mode by $\gamma\mapsto-\gamma$, so the absorption vertex $f$ in $w$ picks up a sign, and then $u_\psi=-u_\one$ follows from $w=\kappa\,v$.

For rigidity, the Ising theory is thus the opposite of the Laughlin state. In the Laughlin state each pair of vortices fuses in a single channel, and the leftover phase does not depend on the pair. Here $\sigma$ is the only vortex, and the leftover phase does not depend on the channel. This is not automatic, since rephasing $\sigma$ multiplies both channels by the same factor and cannot remove a relative phase between them.

Fermion attachment fixes $\sigma$, as for the vortex $3$ of $SO(3)_3$, so here it moves no anyon at all. It has module type $R_\psi$ (Proposition~\ref{prop:alpha}(2)) and acts on every zero mode by $\gamma\mapsto-\gamma$, as the fermion parity of the Majorana modes. So it multiplies the channel $\psi$, which carries odd fermion parity, by $-1$ relative to the channel $\one$. No rephasing removes this sign, because rephasing $\sigma$ multiplies both channels by the same phase, and $\psi$ is not rephased.

\subsection{Comparison}\label{subsec:physics-comparison}

\begin{table}[!ht]
\caption{Fermion-parity gauging and fermion attachment in the three
examples.}\label{tab:examples}
\begin{tabular}{@{}llll@{}}
\toprule
 & Laughlin, $\nu=1/k$ & $SO(3)_3$ & Ising\\
\midrule
fermionic phase & $U(1)_k$ & $SO(3)_3$ & $p+ip$, $SO(N)_1$\\
gauged theory & $U(1)_{4k}$ & $SU(2)_6$ & $\mathrm{Spin}(N)_1$\\
gauging of $(-1)^F$ & $A\mapsto A+c$ & sum over $s$ & sum over $s$\\
vortices & odd $n$ & $1,3,5$ & $\sigma$\\
fermion attachment & $n\mapsto n+2k$ & $1\leftrightarrow5$ & no anyon moves\\
channels of a vortex pair & $1$ & up to $4$ & $2$\\
\botrule
\end{tabular}
\end{table}

We summarize the three examples in Table~\ref{tab:examples}. In the
Laughlin state, fermion parity lies in the electromagnetic $U(1)$ and is
gauged in the Lagrangian, while in $SO(3)_3$ and in the Ising theory it is
gauged by the sum over spin structures. Fermion attachment moves every
vortex of the Laughlin state, some vortices of $SO(3)_3$, and no anyon of
the Ising theory, where it acts only on fusion spaces. Correspondingly, the phases compared by relative rigidity belong to different pairs of vortices in the Laughlin state, to different channels of one pair in the Ising theory, and to both in $SO(3)_3$. In all three cases the lifting ambiguity is the same $\Z/2$.

\backmatter

\section*{Acknowledgements}
R.L. is supported by National Natural Science
Foundation of China under Grant No.~12422503. J.T. is supported by National Natural Science Foundation of China under Grant No.~12405085 and by the Natural Science Foundation of Shanghai (Grant No.~24ZR1419300). JT would like to thank Ying Zhang for her love and support.

\section*{Declarations}

\begin{itemize}
\item \textbf{Competing interests.} The authors declare no competing interests.
\item \textbf{Data availability.} The authors declare that the data supporting the findings of this study are available within the paper.
\item \textbf{Code availability.} Not applicable.
\end{itemize}

\begin{appendices}

\section{An explicit trivialization of \texorpdfstring{$\alpha^2$}{alpha squared}}\label{app:square}

The isomorphism of Corollary~\ref{cor:alpha-square} can be written
explicitly for the representative~\eqref{eq:J-alpha}. This is the
tensor-categorical form of \cite[\S III.A, Eq.~(115)]{Aasen:2021vva},
where the corresponding phases are $\gamma_{a_x}=i^x$.
For homogeneous $X$, define
\[
\beta_X:=
\begin{cases}
\id_X,&p(X)=0,\\[2pt]
r_X\circ(\id_X\otimes e)\circ a_{X,\psi,\psi},&p(X)=1.
\end{cases}
\]
Thus $\beta_X\colon\alpha^2(X)\to X$, and for odd $X$ its source
is $(X\otimes\psi)\otimes\psi$. These maps include all unitors
and associators needed for the two attached fermions to annihilate.

\begin{lemma}\label{lem:square-sign}
For homogeneous $X,Y$ of parities $p,q$,
\[
\beta_{X\otimes Y}\circ J^{\alpha^2}_{X,Y}
=(-1)^{pq}(\beta_X\otimes\beta_Y).
\]
\end{lemma}

\begin{proof}
Label the first and second attached powers in $\alpha^2(X)$ by
$\psi^p_{(1)}$ and $\psi^p_{(2)}$, and similarly for $Y$.
For an even object these are formal unit factors, inserted and
removed by the $r$-maps in~\eqref{eq:J-alpha}. By
\eqref{eq:composite},
\[
J^{\alpha^2}_{X,Y}
=\alpha(J^\alpha_{X,Y})\circ J^\alpha_{\alpha(X),\alpha(Y)}.
\]
Expanding~\eqref{eq:J-alpha}, the second attachment moves
$\psi^p_{(2)}$ past $Y$ and $\psi^q_{(1)}$, while the first
attachment moves $\psi^p_{(1)}$ past $Y$. The hexagon axioms
combine the two crossings past $Y$ into the crossing of
$\psi^p\otimes\psi^p$ past $Y$. This crossing disappears upon
contraction by $\mu_{p,p}$: the precise identity is
\[
r_Y\circ(\id_Y\otimes\mu_{p,p})
\circ c_{\psi^p\otimes\psi^p,Y}
=l_Y\circ(\mu_{p,p}\otimes\id_Y).
\]
It follows from naturality of the braiding and
$c_{\one,Y}=r_Y^{-1}l_Y$.

The remaining difference from separate contraction on $X$ and
$Y$ is the crossing of $\psi^p_{(2)}$ past
$\psi^q_{(1)}$. By supercommutativity~\eqref{eq:mu-comm} and associativity~\eqref{eq:mu-assoc} of $\mu$, the resulting four-power contraction equals $(-1)^{pq}$ times separate contraction of the
two $p$-powers and the two $q$-powers. By the pentagon and triangle axioms, the intervening reassociations agree with those in $\beta_X\otimes\beta_Y$. Hence the stated equality of morphisms from $\alpha^2(X)\otimes\alpha^2(Y)$ to $X\otimes Y$ holds.
\end{proof}

\begin{proposition}\label{prop:square-explicit}
The maps $\zeta_X:=i^{p(X)}\beta_X$ form a monoidal natural
isomorphism $\zeta\colon\alpha^2\Rightarrow\Id_\cE$ with
$\zeta_b=\id_b$ for all $b\in\cC$.
\end{proposition}

\begin{proof}
Naturality of $\beta$ follows from naturality of $a,r$, and
functoriality of $\otimes$. On $\cC$ it is the identity.
By Lemma~\ref{lem:square-sign}, the tensor condition for
$\zeta$ reduces to
\[
i^{p\oplus q}(-1)^{pq}=i^p i^q,
\]
where $p\oplus q$ denotes addition modulo $2$. This is immediate
in all four cases, and for $p=q=1$ it reads $-1=i^2$.
The unit condition holds because $\iota^{\alpha^2}=\id_\one$
and $\zeta_\one=\id_\one$. In particular
$\zeta_\psi=\id_\psi$, since the local fermion belongs to the
even component.
\end{proof}

\end{appendices}

\bibliography{manuscript}

\end{document}